\documentclass[11pt,tbtags,reqno]{amsart}

\usepackage{graphicx}
\usepackage{tabmac}
\usepackage{mathrsfs}
\usepackage{latexsym}
\usepackage{amsthm,amsopn,amsfonts,amssymb,epsfig,stmaryrd}

\makeatletter
\renewcommand{\subsubsection}{\@startsection
{subsubsection}
{3}
{0mm}
{\baselineskip}
{-0.5\baselineskip}
{\normalfont\normalsize\bfseries}}
\makeatother

\newtheorem{theorem}{Theorem}
\newtheorem{lemma}[theorem]{Lemma}
\newtheorem{proposition}[theorem]{Proposition}

\newtheorem{corollary}[theorem]{Corollary}

\newtheorem{definition}[theorem]{Definition}

\newtheorem*{acknow}{Acknowledgments}

\theoremstyle{remark}
\newtheorem{remark}[theorem]{Remark}

\def\la{{\lambda}}

\def\cal L{{\mathcal L}}

\def\aa{\alpha}

\newcommand{\tcercle}[1]{\ensuremath{\setlength{\unitlength}{1ex}\begin{picture}(2.8,2.8)\put(1.4,1.4){\circle{2.7}\makebox(-5.6,0){#1}}\end{picture}}}
\newcommand{\gcercle}{\ensuremath{\setlength{\unitlength}{1ex}\begin{picture}(5,5)\put(2.5,2.5){\circle{5}}\end{picture}}}

\newcommand{\aand}{\ensuremath{ \quad\textrm{and}\quad}}

\def\C{{\mathcal C}}
\def\R{{\mathcal R}}

\def\B{{\mathcal B}}

\let\la\lambda
\let\La\Lambda
\let\Om\Omega

\let\ta\theta

\let\rw\rightarrow

\let\Ga\Gamma
\newcommand{\LL}{\ensuremath{\langle\!\langle}}
\newcommand{\RR}{\ensuremath{\rangle\!\rangle}}

\newcommand{\coeff}[1]{\ensuremath{\underset{#1}{\mathrm{coeff}}}}

\def\cd{{\circledast}}
\def\Sp{\mathrm{{E}}}
\def\Spt{\mathrm{{\widetilde{E}}}}

\def\lrw{\leftrightarrow}
\def\olw{\overleftarrow}
\def\orw{\overrightarrow}

\def\cd{{\circledast}}
\def\Sp{\mathrm{{E}}}
\def\Spt{\mathrm{{\widetilde{E}}}}

\def\lrw{\leftrightarrow}
\def\olw{\overleftarrow}
\def\orw{\overrightarrow}

\begin{document}

\title{{Evaluation} and normalization of Jack superpolynomials}

\author[P.\ Desrosiers]{Patrick Desrosiers}

\address{Instituto de Matem\'atica y F\'{\i}sica, Universidad de
Talca, 2 norte 685, Talca, Chile.}
\email{patrick.desrosiers@inst-mat.utalca.cl}

\author[L.\ Lapointe]{Luc Lapointe}
\address{Instituto de Matem\'atica y F\'{\i}sica, Universidad de
Talca, 2 norte 685, Talca, Chile.}
\email{lapointe@inst-mat.utalca.cl }
\author[P.\ Mathieu]{Pierre Mathieu}
\address{D\'epartement de physique, de g\'enie physique et
d'optique, Universit\'e Laval,  Qu\'ebec, Canada,  G1V 0A6.}
\email{pmathieu@phy.ulaval.ca}

 \begin{abstract}
Two evaluation formulas are derived for the  Jack superpolynomials. The evaluation formulas are   expressed in terms of products of fillings of skew diagrams.  One of these formulas is nothing but the evaluation formula of the Jack polynomials with prescribed symmetry, which thereby receives here a remarkably simple formulation.
Among the auxiliary results required to establish the evaluation formulas, the determination of the conditions ensuring the non-vanishing coefficients in a Pieri-type rule for Jack superpolynomials is worth pointing out.  An important application of the evaluation formulas is a new derivation of the combinatorial norm of the Jack superpolynomials.  We finally mention that the introduction of a simpler version of the dominance ordering on superpartitions is fundamental to establish our results.
\end{abstract}

\thanks{This work was  supported by the Natural Sciences and Engineering Research Council of Canada; the
Fondo Nacional de Desarrollo Cient\'{\i}fico y
Tecnol\'ogico de Chile [\#1090034 to P.D., \#1090016 to L.L.]; and the Comisi\'on Nacional de Investigaci\'on Cient\'ifica y Tecnol\'ogica de Chile [Redes De Colaboraci\'on RED4,
 Anillo de Investigaci\'on ACT56 Lattice and Symmetry].}


 \maketitle

\newpage

{ \small
 \setlength{\parskip}{0pt plus 1pt minus 1pt}

\tableofcontents

 \setlength{\parskip}{5pt plus 1pt minus 1pt}
}

\section{Introduction}
\def\B{\mathcal{B}}
\def\F{\mathcal{F}}
\def\S{\mathcal{S}}

The Jack superpolynomials were introduced in 2003 \cite{DLMcmp2} as the orthogonal eigenfunctions  of a quantum mechanical many-body problem that had been formulated a decade before, namely the supersymmetric Calogero-Moser-Sutherland model \cite{SS,BTW} (see also \cite{DLMnpb} for more results and references on this model).   As their name suggests, the Jack superpolynomials generalize Jack's symmetric polynomials \cite{Mac, Stan} by incorporating both commuting and anticommuting variables.  The presence of anticommuting variables obviously makes
computations and demonstrations more involved than in
the classical theory of symmetric polynomials.
But despite this apparent complexity, the Jack superpolynomials
share many elegant properties with their classical counterparts
\cite{DLMcmp2,DLMadv}, such as orthogonality with respect to
two different scalar products, and duality.
 The aim of the article is to further develop the strong analogy between
the properties of the Jack superpolynomials and those of the Jack polynomials.
Before presenting the most relevant results, let us review some elements of the theory of symmetric superpolynomials.

\subsection{Superpartitions}
Superpartitions were first introduced in 2001 \cite{DLMnpb}, but it was later noticed
{\cite{DLMjalgcomb}}
  that they could be interpreted as overpartitions \cite{CL}
  or as MacMahon standard diagrams \cite{Pak}.  Here we adopt the following definition:
 \begin{definition}  \label{defsuperpart} A superpartition $\Lambda$ of
degree $(n|m)$ and length $\ell$
  is a pair $(\Lambda^\circledast,\Lambda^*)$ of partitions
$\Lambda^\circledast$ and $\Lambda^*$ such
 that
 \begin{enumerate} \item $\Lambda^* \subseteq \Lambda^\circledast$
 \item the degree of $\Lambda^*$ is $n$
 \item the length of $\Lambda^\circledast$ is $\ell$
 \item the skew diagram $\Lambda^\circledast/\Lambda^*$
is both a horizontal and a vertical $m$-strip.
 \end{enumerate}
 \end{definition}
\noindent Note that we follow Macdonald's notation
for partitions, diagrams and skew-diagrams (see Section~\ref{sectintro}
and \cite{Mac}).  Obviously, if
$\Lambda^\circledast= \Lambda^*=\lambda$,
then $\Lambda=(\lambda,\lambda)$ can be interpreted as the partition
$\lambda$.

 A very convenient way to represent superpartitions  was
{introduced}
  in \cite{DLMjalgcomb}. Concretely, the Ferrers diagram of a superpartition $\Lambda=(\Lambda^\circledast,\Lambda^*)$ is obtained by
\begin{enumerate}
 \item drawing the diagram of $\Lambda^\circledast$, and
  \item replacing the cells that belong to
$\Lambda^\circledast/\Lambda^*$ by circles.
  \end{enumerate}
Figure 1 illustrates this procedure for the case
$\Lambda^\circledast=(4,3,3,1,1)$ and $\Lambda^*=(3,3,2,1)$.
To distinguish them from the circles,
the cells corresponding to those of $\Lambda^*$ in the diagram of $\Lambda$
will be called squares.
{Because
  the circles form a horizontal and a vertical strip, two circles cannot appear in the same column nor in the same row. In other words, two rows or two columns ending with a circle cannot have the same length. This situation is clearly reminiscent of the Pauli exclusion principle for fermionic states in quantum physics. For this reason,
   rows and columns that terminates with a circle are called  fermionic, the other ones being said to be bosonic.}
   The bosonic content of a superpartition, denoted by $\mathcal{B}\Lambda$, is defined as the set of squares in the diagram of $\Lambda$ that do not belong at the same time to a fermionic row and
a fermionic column.
The fermionic content of $\Lambda$ is given by the complement of the bosonic content in the diagram $\Lambda^\circledast$, that is,
$\mathcal{F}\La=\Lambda^\circledast/\mathcal{B}\Lambda$.   See Figure 2.

\begin{figure}[h]
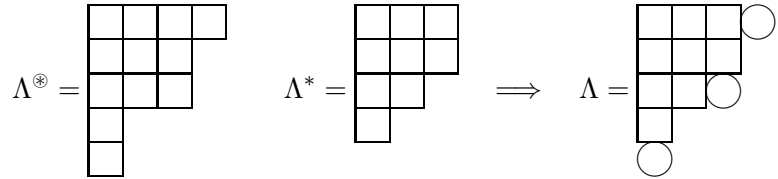
\label{Fig1}\caption{Diagram of a superpartition $\Lambda=(\Lambda^\circledast,\Lambda^*)$}
\begin{equation*}
\Lambda^\circledast={\tableau[scY]{&&&\\&&\\&&\\ \\ \\}}\qquad
\Lambda^*={\tableau[scY]{&&\\&&\\ &\\ \\ \bl}} \quad\Longrightarrow\quad \Lambda={\tableau[scY]{&&&\bl\tcercle{}\\&&\\& &\bl\tcercle{}\\ \\\bl\tcercle{}}}
\end{equation*}
\end{figure}

In the following paragraphs, we shall extensively make use of a partial order on superpartitions
{that generalizes naturally the usual dominance order.}
  Let us recall that for any pair of partitions $\lambda$ and $\nu$ of $n$, $\lambda \geq \nu$ in the dominance order if and only if $\sum_{i=1}^k\lambda_i\geq \sum_{i=1}^k\nu_i$ for all $k$.  We now equip the set of all superpartitions of a given degree $(n|m)$ with the following dominance order:
  \begin{equation}\label{eqdeforder1}
 \Lambda \geq
\Omega   \quad \Longleftrightarrow\quad \Lambda^\circledast \geq \Omega^\circledast \quad \text{ and }\quad
\Lambda^* \geq \Omega^*.
  \end{equation}

\begin{figure}[h]
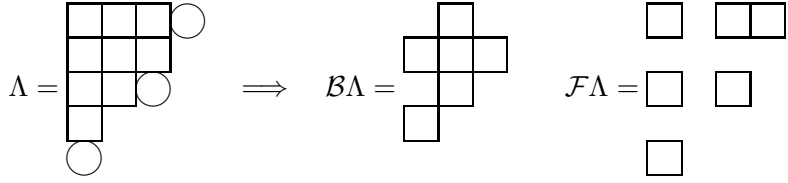
\label{Fig2}\caption{Fermionic and bosonic contents of  a superpartition $\Lambda$}
\begin{equation*}
\Lambda={\tableau[scY]{&&&\bl\tcercle{}\\&&\\& &\bl\tcercle{}\\ \\\bl\tcercle{}}}\quad\Longrightarrow\quad \mathcal{B}\Lambda=
{\tableau[scY]{\bl & \\&&\\ \bl & \\ \\ \bl}}\qquad \mathcal{F}\Lambda={\tableau[scY]{&\bl &&\\ \bl \\ &\bl &\\\bl \\ &\bl}} \qquad
\end{equation*}
\end{figure}

\subsection{Jack symmetric superpolynomials}

Let $x=(x_1,\ldots,x_N)$ and $\theta=(\theta_1,\ldots,\theta_M)$ be two sets of indeterminates that satisfy the following commutation relations:
\begin{equation} x_i x_j=x_jx_i ,\quad x_i \theta_j=\theta_j x_i,\quad \theta_i \theta_j=-\theta_j\theta_i,\quad \theta_i^2=0 \, ,
 \end{equation} for all indices $i,j$.
 A superpolynomial, or polynomial in superspace, is an element of the ring of polynomials in $x$ and $\theta$ over a ring $R$.
 Equivalently, a superpolynomial is an element of the Grassmann algebra generated by $(\theta_1,\ldots,\theta_M)$  over the polynomial ring $R[x_1,\ldots,x_N]$. Following the
 {terminology  used in}
 physics, the variables $x$ and $\theta$ will be respectively called bosonic and fermionic.

 From now on, we set $N=M$ and assume that $R$ is the field $\mathbb{Q}(\alpha)$ of rational functions in the indeterminate $\alpha$.
 A symmetric superpolynomial \cite{DLMnpb} is a superpolynomial $f(x,\theta)$ such that
 \begin{equation}\label{eqdefsymsuperpoly} f(x_1,\ldots,x_N,\theta_1,\ldots,\theta_N)= f(x_{\sigma(1)},\ldots,x_{\sigma(N)},
\theta_{\sigma(1)},\ldots,\theta_{\sigma(N)})
 \end{equation}
 for any  permutation $\sigma$ of $\{1,\ldots,N\}$.  Notice that if
$f(x,\theta)$ is symmetric and homogeneous both in $x$ and in
$\theta$, then
it can be decomposed as follows:
 \begin{equation}\label{Eqformsuperpoly} f(x_1,\ldots,x_N,\theta_1,\ldots,\theta_N)=\sum_{1\leq i_1<\ldots<i_m\leq N}\theta_{i_1}\cdots\theta_{i_m}f_{i_1,\ldots ,i_m}(x), \end{equation}
where $f_{i_1,\ldots ,i_m}(x)$ is a homogeneous polynomials
 antisymmetric in the variables $x_{i_1},\ldots,x_{i_m}$ and symmetric in the
remaining variables.  In fact, $f_{i_1,\ldots ,i_m}(x)$ is an example of
 a polynomial with prescribed symmetry \cite{BDF,Baker,McAnally}.

The set of homogeneous symmetric superpolynomials of degree $n$ in $x$ and
$m$ in $\theta$ obviously forms a finite vector space over $\mathbb{Q}(\alpha)$,
 which will be denoted $\mathscr{R}_N^{n,m}$.
As explained in \cite{DLMnpb,DLMjalgcomb},
 there exists a bijective map between any  basis of $\mathscr{R}^{n,m}_N$ and the set of
superpartitions $\La$ of degree $(n|m)$ and of length not larger than $N$.
 We will see in Section \ref{spins} that the symmetric monomials $m_\Lambda(x,\theta)$ form a simple basis of  $\mathscr{R}^{n,m}_N$.  For the moment,  the only additional information we need
{concerning}
  the symmetric monomials is the following  stability property:
  \begin{equation}\label{restriction} \rho_{M,N} :m_\Lambda(x_1,\ldots,x_{M}, \theta_1,\ldots,\theta_{M}) \mapsto
m_\La(x_1,\ldots,x_{N}, \theta_1,\ldots,\theta_{N})\quad \forall M\geq N,
  \end{equation}
  where $\rho_{M,N}$ is the homomorphism $\mathscr{R}^{n,m}_M\to \mathscr{R}^{n,m}_N$
 that sends the indeterminates $x_{N+1},\theta_{N+1}$, $\ldots$, $x_M,\theta_M$ to zero and acts as the identity on the remaining ones.
 Note that by definition, $m_\La(x,\theta)$ is zero whenever the length of the superpartition is greater than the number of bosonic variables.

The stability property together with the fact that $\rho_{N,N}=\mathrm{id}$ and $\rho_{M,N}\circ\rho_{L,M}=\rho_{L,N}$ for all $N\leq M\leq L$, enable us to take the inverse limit:
\begin{equation}m_\Lambda =\lim_{\longleftarrow}   m_\La(x,\theta)=\left( m_\La(x_1,\theta_1),m_\La(x_1,x_2,\theta_1,\theta_2), m_\La(x_1,x_2,x_3,\theta_1,\theta_2,\theta_3),\ldots\right).
\end{equation}
We can then identify  ${\rm span_{\mathbb Q(\alpha)}}\{m_\La :\, \La \text{ is a superpartition}\}$ with the following bi-graded vector space:
\begin{equation}  \mathscr{R}=\bigoplus_{n\geq 1,m\geq 0}\mathscr{R}^{n,m},\qquad \mathscr{R}^{n,m}=\lim_{\longleftarrow}     \mathscr{R}_N^{n,m}.
  \end{equation}
Given that the symmetric superpolynomials in $N$ bosonic and $N$ fermionic variables form a ring,
the componentwise product between elements of $\mathscr{R}$ is well defined, and so $\mathscr{R}$ also carries the structure of a bi-graded algebra. The elements  of the latter  will be called symmetric superfunctions.
$\mathscr{R}$ is moreover equipped with a surjective homomorphism $\rho_N:\, \mathscr{R}\to \bigoplus_{n,m}\mathscr{R}_N^{n,m}$ that maps all the $x_i$ and $\theta_i$  with $i>N$ to zero.  To sum up,
any element $f$ of $\mathscr{R}$ is a symmetric superfunction; it is equal to a finite linear combination of the monomials
$m_\La$; and to any such $f$  corresponds a symmetric superpolynomial in $N$ bosonic and $N$ fermionic indeterminates, $f(x,\theta)=\rho_N(f)$, which is nonzero if $N$ is large enough.

It was shown in \cite{DLMjalgcomb, DLMadv} that
the algebra $\mathscr{R}$ of symmetric superfunctions can be endowed with a natural scalar product
\begin{equation}
 \LL \quad |\quad \RR\,:\,  \mathscr{R}\times \mathscr{R}\longrightarrow \mathbb{Q}(\alpha),
\end{equation} which generalizes the usual Hall scalar product for symmetric polynomials \cite{Mac} (see \eqref{scalarpp} for an explicit definition of the scalar
product).

We are now in a position to define the Jack superpolynomials.
\begin{definition}\label{defJacksusy}  Let $\Lambda$ be  superpartitions of degree $(n|m)$.
The monic Jack superfunction $P_\Lambda$ is the unique element of $\mathscr{R}$
  that satisfies
 \begin{align} \label{triangdefin}
 & P_\Lambda = m_\Lambda +\sum_{\Omega< \Lambda} c_{\Lambda\Omega} \,m_\Omega  &\text{(triangularity)}  \\
&\LL P_\Lambda | P_\Omega \RR =0  \quad {\rm if~} \La \neq \Om & \text{(orthogonality)}
 \end{align}
where the coefficients $c_{\Lambda\Omega}$ in the triangularity relation belong to
$\mathbb{Q}(\alpha)$.  The monic Jack superpolynomial $P_\La(x,\theta)$ with $N$ bosonic  and  $N$ fermionic indeterminates   is equal to $\rho_N(P_\La)$.
\end{definition}
\noindent The existence of the superpolynomials $P_\Lambda(x,\theta)$ was proved in \cite{DLMadv}\footnote{We stress that in our previous works \cite{DLMcmp2,DLMadv}, we have denoted the monic Jack superpolynomials by $J_\La$. Here we model our notation on the standard one \cite{Mac,Stan} for the monic case and reserve the symbol $J_\La$ for a different normalization -- see Section \ref{speS}.
}.  It was also
 shown in \cite{DLMadv}
that
the  $P_\Lambda (x;\theta)$'s are equivalent
 to the Jack superpolynomials previously defined in \cite{DLMcmp2} as the orthogonal solutions of a quantum mechanical  eigenvalue problem (the orthogonality
being with respect to a distinct scalar product).  Note that the usual Jack symmetric polynomials $P_\lambda(x)$  are recovered by setting $\Lambda=(\lambda,\lambda)$, which corresponds to letting the
degree $m$ in the Grassmann  variables $\theta$ be equal to zero.

To conclude this review section, a precision
is in order.  Definition \ref{defJacksusy} is in fact a slightly more precise version than {the one} presented in \cite{DLMadv} in that the dominance order
controlling the triangular decomposition
is now more restrictive.  Indeed, the partial order $\trianglerighteq$
 used in \cite{DLMadv} was {defined as follows:}  For $\La$ and $\Omega$ two superpartitions of degree $(n|m)$,
\begin{equation} \label{eqorder2}
\La\trianglerighteq \Omega  \qquad \Longleftrightarrow\qquad
\Lambda^* > \Omega^*  \quad \text{or}\quad
\Lambda^* = \Omega^*  \quad \text{and}\quad
\Lambda^\circledast \geq \Omega^\circledast \, ,
\end{equation}
where again the order on partitions is the dominance order.
Observe that the order  $\trianglerighteq$ is {clearly}
less restrictive than the order $\geq$.  We shall nevertheless prove in Appendix~\ref{appen1} that the two orders lead to the same
symmetric polynomials in superspace, which will allow us to exploit all the
properties of the Jack superpolynomials obtained in \cite{DLMcmp2,DLMadv}.

\subsection{{Main} results}

A combinatorial formula for the norm squared $\LL P_\Lambda | P_\La \RR=
\|P_\Lambda \|^2$ was conjectured in \cite{DLMadv}.   With
$\Lambda=(\Lambda^\circledast,\Lambda^*)$
a superpartition of degree $(n|m)$,  the conjecture
given in \cite{DLMadv} is equivalent to
\begin{equation}\label{conj}  \|P_\Lambda \|^2=\alpha^m\prod_{s\in \B\Lambda}
\frac{l_{\Lambda^\circledast}(s)+\alpha\bigl(1+a_{\Lambda^*}(s)\bigr)}
{1+l_{\Lambda^*}(s)+\alpha \, a_{\Lambda^\circledast}(s)} \, ,
\end{equation}
where we
stress that the arm- and leg-lengths are evaluated with respect to two different diagrams (for the definitions of
$a_\lambda(s)$ and $l_\lambda(s)$  we refer to Section~\ref{sectintro}
or \cite{Mac}).

This formula was proved in \cite{LLN} using
a characterization of the Jack superpolynomials in terms of the
non-symmetric Jack
polynomials  -- cf. \cite[Sect.9]{DLMcmp2}. The norm expression is reduced in \cite{LLN} to an identity on partitions whose proof relies on
the Gessel-Viennot lemma.

Here we provide an alternative proof of \eqref{conj}.
Our proof essentially follows Stanley's method \cite{Stan} in which the norm formula for a Jack polynomials $P_\lambda(x)$
is obtained as a consequence of the evaluation formula and the duality property of the $P_\lambda(x)$'s.  In the superpolynomial case,
the proof relies on the duality, established in  \cite[Sect. 6.1]{DLMadv},
and two new evaluation formulas.

The precise statement of these evaluation
formulas requires some more notation.
Let $f(x,\theta)$ be an element  of $\mathscr{R}_N^{n,m}$, that is,  $f(x,\theta)$ is a bi-homogeneous symmetric superpolynomials that can thus be expanded
as in  \eqref{Eqformsuperpoly}.
The evaluation of such symmetric superpolynomials is defined as the map
 \begin{equation}E_{N,m}\,:\, \mathscr{R}_N^{n,m}\longrightarrow \mathbb{Q}(\alpha)
 \end{equation}
such that
\begin{equation}\label{eqdefENm} E_{N,m}( f)=\left[\prod_{1\leq i< j\leq m}(x_i-x_j)^{-1} f_{1,\ldots,m}(x)\right]_{x_1=\ldots=x_N=1}.
\end{equation}
Our central result are the following two evaluation formulas.
\begin{theorem}\label{mainT}Let $\Lambda=(\Lambda^{\circledast}
,\Lambda^*)$
be a superpartition of degree $(n|m)$ such that $\ell(\La)\leq N$.
Let $\mathcal{S}\Lambda$
be the skew-diagram $\Lambda^{\circledast}/\delta_{m+1}$ where $\delta_{m+1}$
stands for the diagram associated to the partition $(m,m-1,\ldots,0)$.  Finally, as in Figure 2, let $\mathcal{B}\Lambda$ denote the bosonic content of $\Lambda$.
Then the evaluation of the monic Jack polynomial $P_\Lambda(x,\theta)$ is given by
\begin{equation} \label{evformula}
E_{N,m}(P_\Lambda)=\frac{\prod_{s\in \mathcal{S} \Lambda }
\left(N-l_{\La^\circledast}'(s) +\alpha  a_{\La^\circledast}'(s)\right)}{\prod_{s\in \B \Lambda}\left(1+l_{\Lambda^*}(s)+\alpha
a_{\Lambda^{\circledast}}(s)\right)}.\end{equation}
\end{theorem}
\begin{theorem}\label{mainT2}Let $\Lambda=(\Lambda^{\circledast}
,\Lambda^*)$
be a superpartition of degree $(n|m)$ such that $m>0$ and
$\ell(\La) \leq N$.
Let $\tilde{\mathcal{S}} \Lambda$
be the skew-diagram $\Lambda^{*}/\delta_{m}$.
The evaluation of
\begin{equation}F_\Lambda=\left[(-1)^{m-1}\partial_{\theta_N} P_\Lambda(x,\theta)\right]_{x_N=0}
\end{equation}
 is given by
\begin{equation}
E_{N-1,m-1}\bigl(F_\Lambda\bigr)=
\frac{\prod_{s\in \tilde \S\Lambda }
(N-1-l_{\La^*}'(s) +\alpha  a_{\La^*}'(s))}{\prod_{s\in \B\Lambda}(1+l_{\Lambda^*}(s)+\alpha
a_{\Lambda^{\circledast}}(s))}.\end{equation}
\end{theorem}

Let us emphasis some unusual aspects of the evaluation.
We first stress that in the evaluation  of a superpolynomial, only the commuting variables $x_i$ are specialized at 1. Clearly, the anticommuting variables cannot  be set equal to a common anticommmuting value since
 every fermionic monomial of degree larger than 1 would then vanish.
The necessity of factorizing a Vandermonde determinant is also easily understood.
A homogeneous symmetric superpolynomial is of the form \eqref{Eqformsuperpoly} where $f_{i_1\cdots i_m}$ is antisymmetric with respect to $x_{i_1},\cdots,x_{i_m}$ so that these variables cannot be set equal to 1
without causing the direct vanishing of the whole expression. Hence, before specializing each term,  one has to factorize its antisymmetric core, that is, divide it by a Vandermonde determinant of order $m$.

\subsection{{Organization of the article}}
Before plunging into the different steps leading to
the proof of Theorems \ref{mainT} and \ref{mainT2},
we need to review further results concerning superpartitions and symmetric superpolynomials. This is the subject of Section 2.
The derivation of the evaluation formula relies on establishing in Section \ref{SLR} the necessary conditions for the non-vanishing of the Pieri-type  coefficients. The relevant results in that regard are Propositions~\ref{PropgI} and \ref{PropgII} (proved in Appendix~\ref{appenB}).

Another required new tool is what might be called the analogue of the ``column-by-column'' decomposition of a Jack polynomial (cf. \cite[Prop. 5.1]{Stan}).
In the present context, where a column might be either fermionic or bosonic, this requires the introduction of two distinct operations described in
Section~\ref{facpro}: the stripping of a bosonic column and the
transmutation of a fermionic column into a bosonic one
{(see Figure~\ref{FigC}).
At the core of these column decompositions is the following remarkable
property:
removing/transmuting a leftmost column of a Jack superpolynomial in the right number of variables generates another  Jack superpolynomial, up to a proportionality factor in the non-monic case.
These factors are the building blocks of the expression for the combinatorial norm as shown in Section \ref{Snorm}. Such  proportionality factors,
being the ratio of two polynomials, are most readily computed when
the polynomials are specialized to particular values of their variables.

The proof of the evaluation formula given in Theorem \ref{mainT} is presented in Section \ref{spe1}.
As explained above, before implementing the evaluation, one must first divide by a Vandermonde determinant of order $m$. Remarkably, when $m>0$
this order can be reduced from $m$ to $m-1$, which leads to the
second non-trivial evaluation formula given in Theorem~\ref{mainT2} and
whose
proof
is
presented in Section~\ref{spe2}.

As an aside, we mention that before obtaining the evaluation formula
\eqref{evformula} expressed in terms of skew diagrams,
a quite different-looking version
had been obtained by experimentation.
Since this might be of independent interest, it is presented in
Appendix~\ref{appD},
where the connection between the two formulas is also sketched.

{
Finally, it
should be clear from the remark following Theorem \ref{mainT2} that the
evaluation formula \eqref{evformula}
for Jack polynomials in superspace is actually an
evaluation formula for ordinary Jack polynomials with mixed symmetry
(or with prescribed symmetry in the terminology of \cite{BDF,Baker,McAnally}).
This implies that our evaluation formula must agree with the
one presented in  \cite[Prop. 3.6]{McAnally} (yet another expression is given in \cite{Dunkl98}).
It is remarkable that the very complicated looking-form of the latter
evaluation formula can be reexpressed in the simple form presented here.}

\begin{acknow}
 This work was  supported by the Natural Sciences and Engineering Research Council of Canada; the
Fondo Nacional de Desarrollo Cient\'{\i}fico y
Tecnol\'ogico de Chile [\#1090034 to P.D., \#1090016 to L.L.]; and the Comisi\'on Nacional de Investigaci\'on Cient\'ifica y Tecnol\'ogica de Chile [Redes De Colaboraci\'on RED4,
 Anillo de Investigaci\'on ACT56 Lattices and Symmetry].
\end{acknow}

\section{Definitions} \label{sectintro}
Let us first recall some definitions
related to partitions \cite{Mac}.
A partition $\lambda=(\lambda_1,\lambda_2,\dots)$ of degree $d$
is a vector of non-negative integers such that
$\lambda_i \geq \lambda_{i+1}$ for $i=1,2,\dots$ and such that
$\sum_i \lambda_i=d$.  The length $\ell(\lambda)$
of $\lambda$ is the number of non-zero entries of $\lambda$.
Each partition $\lambda$ has an associated Ferrers diagram
with $\lambda_i$ lattice squares in the $i^{th}$ row,
from the top to bottom. Any lattice square in the Ferrers diagram
is called a cell, where the cell $(i,j)$ is in the $i$th row and $j$th
column of the diagram.  Given  a partition $\lambda$, its
conjugate $\lambda'$ is the diagram
obtained by reflecting  $\lambda$ about the main diagonal.
Given a cell $s=(i,j)$ in $\lambda$, we let
\begin{equation}
a_{\lambda}(s)=\lambda_i-j\, , \qquad   a'_{\lambda}(s)=j-1 \, , \qquad
l_{\lambda}(s)=\lambda_j'-i \, ,   \quad  {\rm and} \quad l_{\lambda}'(s)=i-1  \, .
\end{equation}
The quantities $a_{\lambda}(s),a_{\lambda}'(s),l_{\lambda}(s)$ and $l_{\lambda}'(s)$
are respectively called the arm-length, arm-colength, leg-length and
leg-colength.  For instance, if $\lambda=(8,5,5,3,1)$
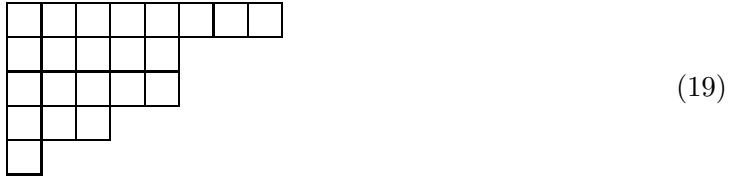
\begin{equation}
{\tableau[scY]{&&&&&&&\\&&&&\\&&&&\\& &  \cr \cr }}
\end{equation}
we have that $a_{\lambda}(3,2)=3,a_{\lambda}'(3,2)=1,
l_{\lambda}(3,2)=1$ and $l_{\lambda}'(3,2)=2$.
We say that the diagram $\mu$ is contained in $\la$, denoted
$\mu\subseteq \la$, if $\mu_i\leq \la_i$ for all $i$.  Finally,
$\la/\mu$ is a horizontal (resp. vertical) $n$-strip if $\mu \subseteq \lambda$, $|\lambda|-|\mu|=n$,
and the skew diagram $\la/\mu$ does not have two cells in the same column
(resp. row).

We now review some basic results concerning superpartitions and the objects for which they provide the proper labeling, namely the
symmetric polynomials in superspace. The functions of interest here are the superspace version of the Jack polynomials, which are
introduced in Section~\ref{Ssjack}.
This material is essentially lifted from \cite{DLMnpb,DLMcmp2,DLMjalgcomb,DLMadv}.

\subsection{Operations on superpartitions}\label{Spp}
Let us first go back to Definition \ref{defsuperpart}.
 When considering the superpartition $\Lambda$ as a diagram such as
in Figure 1,
$\Lambda^\circledast$ corresponds to the diagram obtained by replacing the circles in $\Lambda$ by cells.  Similarly, $\Lambda^*$ corresponds to the diagram obtained by removing all the circles in $\Lambda$.
This allows us to consider the circled star $\circledast$ and the star $*$ as
operations on superpartitions (see Figure 3).
  \begin{figure}[h]
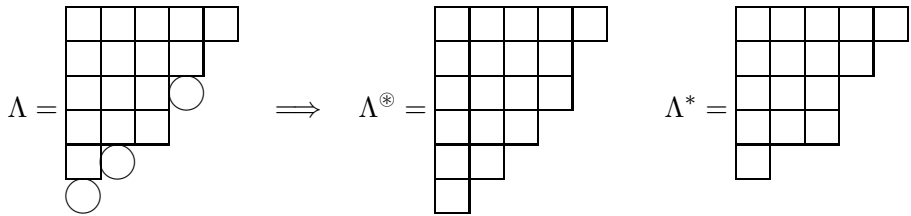
\caption{Operations $\circledast$ and $*$ on a superpartition $\La$}
\begin{equation*}
\La={\tableau[scY]{&&&&\\&&&\\&&&\bl\tcercle{}\\&&\\&\bl\tcercle{}\\ \bl\tcercle{}}}\quad\Longrightarrow\quad
 \La^\circledast={\tableau[scY]{&&&&\\&&&\\&&&\\&&\\&\\ &\bl}}\qquad \La^*= {\tableau[scY]{&&&&\\&&&\\&&&\bl\\&&\\
    &\bl\\ \bl}}
\end{equation*}
\end{figure}

The bosonic
degree $|\La|$ of the superpartition $\La$
is equal to $|\Lambda^*|$ (the number of cells in the diagram
of $\Lambda^*$).  The fermionic degree
$\underline{\overline{\Lambda}}$ of $\Lambda$
is the number of circles in the diagram of $\Lambda$, that is,
$\underline{\overline{\Lambda}}=|\Lambda^{\circledast}|-|\Lambda^*|$.
We say that $\La$ is
a superpartition of degree $(n|m)$ if
$|\La|=n$ and $\La$ has fermionic degree
$m$.
The length $\ell(\Lambda)$ of the superpartition $\Lambda$
is equal to the length of $\Lambda^{\circledast}$ (the number of rows
in the diagram of $\Lambda^{\circledast}$).

Though very practical for many purposes, such as to
define the dominance order \eqref{eqdeforder1} on superpartitions,
Definition ~\ref{defsuperpart}  turns out to be less effective when  working directly on symmetric superpolynomials.   This is why we shall occasionally return to the original definition of a superpartition given in \cite{DLMnpb}.
\begin{definition}\label{defsuperpart2} A superpartition $\La$ of length $\ell$ is
a pair of partitions $(\La^a; \La^s)$, the first one of which contains at most one 0 and does
not have repeated entries.  Explicitly,
\begin{equation}\label{sppa}
\La=(\La^{a};\La^{s})=(\La_1,\ldots,\La_m;\La_{m+1},\ldots,\La_\ell),
\end{equation}
where
\begin{equation}\label{sppb}
\La_1>\ldots>\La_m\geq0 \quad  \text{ and}
\quad \La_{m+1}\geq \La_{m+2} \geq \cdots \geq
\La_\ell > 0 \, .\end{equation}
\end{definition}
\noindent Note that $m$ corresponds in this definition to
the fermionic degree of $\Lambda$. When
$m=0$, we simply omit the semi-column in $\La=(\emptyset;\La^{s})$
and identify $\La$ with $\Lambda^s$.

The equivalence between the two definitions is quite obvious: {the parts of $\La$ that belong to $\La^a$ are the parts of $\La^*$ such that
$\La^\circledast_k-\La^*_k=1$.}
Going back to the example given in Figure 1, we see that if $\La$ is such that $\La^\circledast=(4,3,3,1,1)$  and $\La^*=(3,3,2,1)$, then we have $\Lambda=(3,2,0;3,1)$.
It is clear that $a:\, \La\mapsto\La^a $ can be viewed as a map that sends superpartitons of degree $(n|m)$ to strictly decreasing partitions of length $m$ and with at most one part equal to zero.  In the same vein, $s:\,\La\mapsto\La^s$ maps superpartitions of degree $(n|m)$ end length $\ell$ to partitions of length $n-m$.

 We finally define an important involution on the set of superpartitions:
the conjugation. It is actually simpler to define the conjugation {diagrammatically}:  the  conjugate of a
superpartition $\La$, denoted by $\La'$, is the superpartition whose
diagram is obtained by reflecting the
diagram of $\La$ with respect to the main diagonal.
As shown in Figure 4, reflecting for instance the diagram of $\La=(3,1,0;5,4,3)$ gives $\La'=(5,4,2;4,1)$.

\begin{figure}
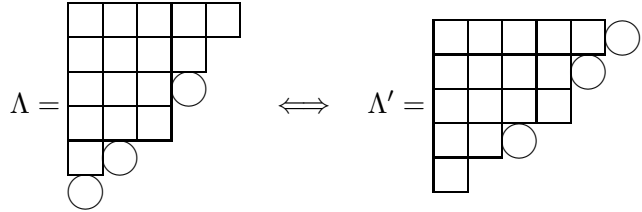
\label{figconj}\caption{Conjugation map on a superpartition $\Lambda$}
\begin{equation*} \label{exdia2}
\La={\tableau[scY]{&&&&\\&&&\\&&&\bl\tcercle{}\\&&\\&\bl\tcercle{}\\ \bl\tcercle{}}}\quad \Longleftrightarrow \quad \La'={\tableau[scY]{&&&&&\bl\tcercle{}\\&&&&\bl\tcercle{}\\&&&\\&&\bl\tcercle{}\\&\bl
}}
\end{equation*}
\end{figure}

\subsection{Classical bases for symmetric polynomials in superspace}\label{spins}As already mentioned in the introduction, a polynomial in superspace (a superpolynomial for short) is a  polynomial in $x=(x_1,\ldots,x_N)$ and $\theta=(\theta_1,\ldots,\theta_N)$, where $\theta$ denotes a set of $N$ Grassmann variables. Rephrasing the information contained in \eqref{eqdefsymsuperpoly}, a polynomial in superspace
is said to be symmetric if it is invariant under the simultaneous interchange of $x_i\lrw x_j$ and $\ta_i\lrw\ta_j$ for all $i,j$.
The set of all homogeneous symmetric polynomials of degree $n$ in $x$ and degree $m$ in $\theta$, forms a
vector space $\mathscr{R}_N^{n,m}$ over $\mathbb{Q}(\alpha)$.

In what follows, we adopt the notation of Definition \ref{defsuperpart2}
and suppose
that the superpartition $\Lambda=(\La_1,\ldots,\La_m;\La_{m+1},\ldots,\La_\ell)$
is of
degree $(n|m)$ and length $\ell$.  We always assume $\ell\leq N$.  In the case where $\ell<N$, we set:
\begin{equation}\La_{\ell+1}=\ldots=\La_N=0.\end{equation}

A simple basis for the space $\mathscr{R}_N^{n,m}$  is furnished by the following extension to superspace of the usual monomial symmetric functions $m_\la(x)$:
\begin{equation} \label{eqmono}
m_\La(x;\theta)=\frac{1}{n_\La!}\sum_{\sigma \in S_N} \theta_{\sigma(1)}\cdots\theta_{\sigma(m)} x_{\sigma(1)}^{\La_1}\cdots x_{\sigma(N)}^{\La_N} \, ,
\end{equation}
where $\La$ is a superpartition of degree $(n|m)$ and of length $\ell(\La)\leq N$,  and
\begin{equation}n_\La!:=\prod_{i\geq1}n_{\La^s}(i)!,\end{equation}
with $n_{\La^s}(i)$
being the number of parts in $\La^s=(\La_{m+1},\ldots,\La_N)$
that are equal to $i$.
The factor $n_\La!$ {is introduced to guarantee
that  the distinct non-symmetric monomials} of the form $\theta_{i_1}\cdots\theta_{i_m}x_{i_1}^{\La_1}\cdots x_{i_N}^{\La_N}$ appear in $m_\La$ with coefficients
equal to $\pm 1$.  As explained in the introduction, the set of all monomials $\displaystyle m_\La=\lim_{\longleftarrow}m_\La(x;\theta)$ forms a linear basis for the algebra $\mathscr{R}$ of symmetric superfunctions.

Another basis of symmetric superpolynomials is given by the power-sums
\begin{equation}
p_\La(x,\theta):=\tilde{p}_{\La_1}(x, \ta)\cdots\tilde{p}_{\La_m}(x, \ta)\, p_{\La_{m+1}}(x)\cdots p_{\La_\ell}(x),\end{equation}
 where
\begin{equation} \label{scalarpp}
\tilde{p}_n(x, \ta):=\sum_i\theta_ix_i^n\qquad\text{and}\qquad p_n(x):=
\sum_ix_i^n \, .
\end{equation}
Now let $p_\La$ denote the inverse limit of the superpolynomial $p_\La(x,\theta)$.   Then $\mathscr{R}$ is equal to
${\rm span}_{\mathbb Q(\alpha)}\{p_\La:\, \La \text{ is a superpartition}\}$.
The relevance of the power sums $p_\La $ in this article
is rooted in the natural scalar product on $\mathscr{R}$ defined as
\begin{equation} \label{defscalprodcombI} \LL \,
{p_\La} \, | \, {p_\Om }\, \RR= {(-1)^{\binom{m}{2}} }\, \alpha^{{\ell}(\La)}\, z_{\La^s}
\delta_{\La,\Om}\,,
 \end{equation}
 where
\begin{equation}
z_{\La^s}=\prod_i i^{n_{\La^s}(i)} {n_{\La^s}(i)!}\, .
\end{equation}
The sign $(-1)^{\binom{m}{2}}$ arises in all scalar products of
symmetric superfunctions of fermionic degree $m$.   It is thus convenient
to define:
\begin{equation}
\overleftarrow{F} = (-1)^{\binom{m}{2}} F
\end{equation}
on any homogeneous superfunction
 $F $ of fermionic degree $m$.  In fact, the left-arrow is the involution in the Grassmann algebra generated by $\theta$ that reverses the order of the variables
 $\theta$, that is, $ \overleftarrow{\theta_{i_1}\cdots \theta_{i_m}}=\theta_{i_m}\cdots\theta_{i_1}$.
For cosmetic reasons, we also introduce
\begin{equation}
\overrightarrow{F } =  F
\end{equation}
which allows to {write} the scalar product \eqref{defscalprodcombI}
in a symmetrical fashion:
\begin{equation} \label{defscalprodcomb} \LL \,
\overleftarrow{p_\La} \, | \, \overrightarrow{p_\Om }\, \RR=\alpha^{{\ell}(\La)}\, z_{\La^s}
\delta_{\La,\Om}\,.
 \end{equation}

\def\vs{\varsigma}

We shall also make use of the elementary superpolynomials $e_\Lambda(x,\theta)$, which provide another  multiplicative basis for $\mathscr{R}^{n,m}_N$
  They are defined as follows:
\begin{equation}
e_\La (x, \ta):=\tilde{e}_{\La_1}(x, \ta)\cdots\tilde{e}_{\La_m}(x, \ta)\,e_{\La_{m+1}}(x)\cdots e_{\La_\ell}(x),\end{equation}
where $\Lambda$ is again a superpartition of fermionic degree $m$ and length
$\ell(\La)= \ell$,
and
where
\begin{equation}
\tilde{e}_n(x, \ta):= m_{(0;1^n)}(x,\ta) \qquad\text{and}\qquad e_n(x):=
m_{(1^n)}(x).\end{equation}

\subsection{Jack polynomials in superspace}\label{Ssjack}

The basis of symmetric polynomials in superspace of concern here
 is the
 generalization of the Jack polynomials. They
are most naturally defined as solutions of a double eigenvalue problem
\cite{DLMcmp2,DLMadv}.  Theorems 22 and 31 in \cite{DLMcmp2} together with the discussion in Appendix \ref{appen1} readily establish the following.

\begin{theorem}\label{TheoEigenJack}Let $D$ and $\Delta$
be the two following algebraically
independent and commuting differential operators:
 \begin{equation}
 D= \frac{1}{2}\sum_{i=1}^N  \alpha x_i^2\partial_{x_i}^2
+\sum_{1 \leq i\neq j \leq N}\frac{x_ix_j}{x_i-x_j}\left(\partial_{x_i}-\frac{\theta_i-\theta_j}{x_i-x_j}\partial_{\theta_i}\right),
\end{equation}
 and
 \begin{equation}
 \Delta= \sum_{i=1}^N \alpha x_i\theta_i\partial_{x_i}\partial_{\theta_i}+
\sum_{1 \leq i\neq j \leq N}
\frac{x_i\theta_j+x_j\theta_i}{x_i-x_j}\partial_{\theta_i}. \end{equation}
Let also
\begin{equation}\label{Dvap}
\varepsilon_\La(\alpha)= \alpha b(\La')-b(\La),\end{equation}
where $b(\La)=\sum_{i=1}^{\ell(\La^*)}(i-1)\La^*_i$. Finally, let
 \begin{equation}\label{epsi}  \qquad\epsilon_\La(\alpha)=\alpha|\La^a|-|{\La'}^a|.\end{equation}  Then, there exists a unique  monic
symmetric polynomial in superspace,
\begin{equation} \label{Ptriangular}
P_\La(x,\theta)=m_\La(x,\theta)+\sum_{\Om<\La}c_{\La\Om}(\alpha)\,m_\Om(x,\theta),
\end{equation}
satisfying
\begin{equation}
D\,P_\La(x,\theta)=
\varepsilon_{\La}(\alpha)\,P_\La(x,\theta)\qquad \text{and}\qquad \Delta\,P_\La(x,\theta)=\epsilon_\La(\alpha)\,P_\La(x,\theta).
\end{equation}
\end{theorem}

 As is the case for the Jack polynomials,
the coefficients $c_{\La\Om}$ in the expansion \eqref{Ptriangular} of $P_{\La}$ do not depend on the number of variables.
It then easily follows that the $P_\La$'s behave well under the obvious extension $\rho_{N,N-1}:\, \mathscr{R}^{n,m}_N \to \mathscr{R}^{n,m}_{N-1}$ of the standard  homomorphism that restricts the number of variables
(see \cite[p. 18]{Mac} and the introduction).
This allow us to take the inverse limit of any Jack superpolynomial and obtain a {\it Jack symmetric superfunction $P_\La$}.  In other words, it makes sense to work with  $P_\La$ even if the  latter contains an infinite number of variables since it is equal to a finite and stable sum of monomials $m_\La$.

As mentioned in the introduction, Jack superfunctions have been shown to be
orthogonal with respect to the scalar product (\ref{defscalprodcomb}) \cite{DLMadv}:
\begin{equation} \label{ortho} \LL \olw{P_\La} \, | \, \orw{P_\Om }\RR=
 \, \|P_\Lambda \|^2
\delta_{\La,\Om},\end{equation}
with $\|P_\Lambda \|^2\neq 0$ a certain rational function in $\alpha$
that does not depend on the number of variables $N$
(the non-vanishing of $\|P_\Lambda \|^2$ follows from the fact that
the scalar product \eqref{defscalprodcomb} is
positive definite when $\alpha > 0$).

Directly related to {this} orthogonality relation, we have the Cauchy formula
\begin{equation}
\prod_{i,j}\frac{1}{(1-x_iy_j-\theta_i\phi_j)^{1/\alpha}}\,
=\sum_{\La}\frac{1}{\|P_\Lambda \|^2}
\olw{P_\La}\!(x,\theta)\orw{P_\La}\!(y,\phi)\,,
\end{equation}
where $\|P_\Lambda \|^2$ was defined in \eqref{ortho}.

We conclude this section by mentioning  a
useful duality property of $P_\La$.
Let $\hat \omega_{\alpha}$ stand for the endomorphism of the space
of symmetric polynomials in superspace
defined on the power
sums as
\begin{equation}
\hat \omega_{\alpha} (p_n) = (-1)^{n-1} \alpha \, p_n \qquad {\rm and}
\qquad \hat \omega_{\alpha} (\tilde p_n) = (-1)^{n} \alpha \, \tilde p_n.
\end{equation}
It was shown in \cite[Theo. 27]{DLMadv} that
\begin{equation}\label{dual}
\hat \omega_{\alpha} (\orw{P_{\La} })
= \|P_\Lambda \|^2\, \olw{P_{\La'}}^{(1/\alpha)},
\end{equation}
where $P_{\La'}^{(1/\alpha)}$ stands for $P_{\La'}$ with $\alpha$ replaced
by $1/\alpha$.

\subsection{{Complementary remarks on the eigenfunction characterization of the Jack superpolynomials}}

For completeness,
we provide some clarification comments on the description of the Jack superpolynomials as   eigenfunctions of a quantum $N$-body problem.   {In that regard, we clear up some discrepancies between the notations used in the current paper and those of previously-quoted articles.} None of these comments is used in the rest of the article, so that this subsection can be safely skipped.

The version of Theorem~\ref{TheoEigenJack} presented in
\cite[Theo. 14]{DLMadv}
differs slightly from the one presented here.  The eigenvalue problem
in \cite{DLMadv} is given in terms of operators $\mathcal H$
and $\mathcal I$ that are related to $D$ and $\Delta$ through the
relations
\begin{equation}\label{DDrel}
D = \frac{1}{2}\alpha {\mathcal{H}}-\frac{(\alpha+N-1)}{2} {\mathcal H}_1
\qquad
{\rm and} \qquad
\Delta= \alpha \mathcal{I} +\frac{1}{2} ({\mathcal I}_0^2-{\mathcal I}_0),
\end{equation}
where
\begin{equation}\label{h1etio} {\mathcal H}_1=\sum_{i=1}^N x_i \partial_{x_i}\qquad
{\rm and} \qquad{\mathcal I}_0 =  \sum_{i=1}^N \theta_i \partial_{\theta_i}.
\end{equation}  Given that ${\mathcal H}_1$ and
${\mathcal I}_0$ are constant on polynomials in superspace of a given fermionic
and bosonic degree, the theorem still holds (after an obvious modification
of the eigenvalues).

In the $\theta\rightarrow0$ limit, $D$ becomes the operator $D^{\text{S}}$
used in  \cite[Eq. 11]{Stan}, up to a rescaling and minor modifications
that remove
the dependence upon $N$ in the eigenvalues:
\begin{equation}
\lim_{\ta\to0}D=2D^{\text{S}}-2(N-1){\mathcal H}_1,
\end{equation}
where ${\mathcal H}_1$ is defined in (\ref{h1etio}).
In physics, $D$ is interpreted as the Hamiltonian (energy operator).
The operator $\Delta$ is a conserved quantity that disappears in
the non-supersymmetric case (i.e., when $\theta_i\rw 0$).

The Hamiltonian and the conserved quantity just mentioned  refer to the  supersymmetric extension of the trigonometric  Calogero-Moser-Sutherland (stCMS) model \cite{SS,BTW,DLMnpb}.
Let us digress for a moment to comment on this a priori curious situation that eigenfunctions of the stCMS Hamiltonian only (that is, the eigenfunctions of $D$) can fail to be orthogonal. It is clear form \eqref{Dvap} that the Hamiltonian eigenvalues are insensitive to the fermionic or bosonic nature of the parts in the superpartition parametrizing the eigenfunction. There is thus a residual degeneracy. The way to lift this degeneracy is, however, clear from the point of view of integrable systems. Recall that the usual trigonometric CMS model, being integrable, has $N$ (the number of degrees of freedom) independent and mutually commuting conservation laws -- the Hamiltonian being one of them. But since the stCMS model has $2N$ degrees of freedom, it must have an extra set of $N$ commuting conservation laws -- disappearing when the anticommuting variables are set equal to zero. Select the first non-trivial representative of this second tower, called the partner Hamiltonian (this is essentially $\Delta$). The common eigenfunctions of the Hamiltonian and its partner turn out to be have non-degenerate eigenvalues; in addition, they are orthogonal  \cite{DLMcmp2}. These are the Jack polynomials in superspace.
}

\section{Linear expansion of products of Jack superpolynomials}\label{SLR}

{As a preliminary step toward the derivation of} the evaluation formula for the Jack superpolynomials,
{the following two technical problems must be addressed:}
\begin{enumerate}
\item {Identify} the Jack superpolynomials that can appear
in the expansion of $P_{R} \cdot P_\La$, {where} $R$ is a single {row  (bosonic or fermionic) superpartition.}
\item {Obtain} $P_\La(x_1,x_2,\ldots,y_1,y_2\ldots;\theta_1,\theta_2,\ldots,\phi_1,\phi_2,\ldots)$ as a linear combination of Jack superpolynomials in $x$ and $\theta$ with coefficients in $y$, $\phi$, and $\alpha$.
\end{enumerate}
In Sections \ref{sectgenvc} and \ref{sectstrips} below,  we   address the first {point} by carefully studying the coefficients $g_{\Om,\Gamma}^\La$, which are defined as
the rational function in $\alpha$ satisfying
\begin{equation}\label{EqProdJ}
P_\Om \, P_\Gamma= \sum_\La \frac{1}{{\|P_\La\|^2}} \,g_{\Om,\Gamma}^\La\, P_\La.
\end{equation}
By orthogonality, the latter equation is equivalent to
\begin{equation}
g_{\Om,\Gamma}^\La =\LL \olw{P_\La}   |  \orw{P_\Om}\orw{P_\Gamma} \RR.\end{equation}
{The second issue is  considered}
 in Section \ref{Sskewjack}, where we  introduce skew Jack polynomials in superspace by using the coefficients
$g_{\Om,\Gamma}^\La$.

\subsection{Necessary conditions for the non-vanishing of
$g_{\Om,\Gamma}^\La$} \label{sectgenvc}
The following lemma is an immediate consequence of the duality \eqref{dual} induced by $\omega_{\alpha}$.
\begin{lemma} \label{lemmaequivconj}
We have that
\begin{equation}\label{lessi}
g^{\La}_{\Om \Gamma}\neq 0\quad\text{ if and only if }
\quad g^{\La'}_{\Om' \Gamma'}\neq 0.
\end{equation}
\end{lemma}

Our first non-trivial result on the coefficients $g_{\Om,\Gamma}^\La$ shows
 that they behave quite like their non-superspace counterparts.
It {also neatly} illustrates the efficiency of the ordering
\eqref{eqdeforder1} on
superpartitions.
Recall from \cite[p.5-6]{Mac}
that given two partitions $\lambda$ and $\mu$,
$\lambda \cup \mu$ stands for the partition whose parts are those
of $\lambda$ and $\mu$ arranged in weakly decreasing order,
while $\lambda+ \mu$ stands for the partition whose $i^{th}$ part is
$\lambda_i+\mu_i$.  The two notions are related by the formula
$(\lambda \cup \mu)'=\lambda'+\mu'$.
The following proposition is a direct generalization
of \cite[Prop. 4.1]{Stan}.

\begin{proposition}\label{Propg}
If $g_{\Om,\Gamma}^\La \neq 0$ then
\begin{equation}
\Om^* \cup \Gamma^* \leq \La^*
\leq \Om^* + \Gamma^* \qquad {\rm and} \qquad
 \Om^{\circledast} \cup \Gamma^{\circledast} \leq \La^{\circledast}
\leq \Om^{\circledast} + \Gamma^{\circledast}.
\end{equation}
Moreover, if there exists a superpartition
$\Lambda$ such that
$\La^*= \Om^* \cup \Gamma^*$
and  $\La^{\circledast}=\Om^{\circledast} \cup \Gamma^{\circledast}$,
then   $g_{\Om,\Gamma}^\La \neq 0$. Similarly, if
there exists a superpartition
$\Lambda$
such that
$\La^*= \Om^* + \Gamma^*$
and  $\La^{\circledast} =\Om^{\circledast} + \Gamma^{\circledast}$, then
$g_{\Om,\Gamma}^\La \neq 0$.
\end{proposition}
\begin{proof}  We first prove that
\begin{equation}\label{Eqmm}
m_{\Om} \, m_{\Gamma} = \sum_{\Lambda} t^{\La}_{\Om \Gamma} \, m_{\La}
\end{equation}
is such that $t^{\La}_{\Om \Gamma}$ is non-zero only if
$\La^* \leq \Om^* + \Gamma^*$ and  $\La^{\circledast}
\leq \Om^{\circledast} + \Gamma^{\circledast}$.  Furthermore, we prove
that if there exists a superpartition
$\Lambda$ such that
$\La^*= \Om^* + \Gamma^*$
and  $\La^{\circledast}=\Om^{\circledast} + \Gamma^{\circledast}$,
then
$t^{\La}_{\Om \Gamma}$ is non-zero.

Since $\olw{e_{\La'}} = \lim_{\alpha \to 0} \orw{P_{\La}}$
\cite[Eq. 6.22]{DLMadv}
we have from \eqref{Ptriangular} that
\begin{equation} \label{Eqeenm}
e_{\Lambda'} = \pm m_{\La} + \sum_{\Om < \La} w_{\Om \La} \, m_{\Lambda},
\end{equation}
which implies that
\begin{equation}
m_{\Lambda} = \pm e_{\La'} + \sum_{\Omega < \La} w_{\Om \La}' e_{\Om'}.
\end{equation}
This immediately gives that
\begin{equation} \label{eqapres}
m_{\Om} \, m_{\Gamma} =
\sum_{\Delta \leq \Om, \Upsilon \leq \Gamma } w_{\Om \Delta}' w_{\Gamma \Upsilon}' e_{\Delta'} e_{\Upsilon' }
\end{equation}
where the coefficient of $e_{\Om'} \, e_{\Gamma'}$ is equal to $\pm 1$.
Now, $e_{\Delta'} e_{\Upsilon' }=0$ if $\Delta'$ and $\Upsilon'$ have fermionic rows of the
same lengths. Otherwise $e_{\Delta'} e_{\Upsilon' }=\pm e_{\Theta'}$, where
$\Theta$ is the unique superpartition such that $\Theta^*=\Delta^*+ \Upsilon^*$
and $\Theta^{\circledast}=\Delta^{\circledast}+ \Upsilon^{\circledast}$.
This implies, from \eqref{Eqeenm}, that if $m_{\La}$ appears in
$e_{\Delta'} e_{\Upsilon' }$ then $\La \leq \Theta$.
Hence, if $m_{\La}$ appears in
$m_{\Om}\, m_{\Gamma}$ then
\begin{equation}
\La^* \leq \Theta^* \leq \Delta^*+ \Upsilon^*\leq \Om^*+\Gamma^* \qquad
{\rm and }\qquad
\La^{\circledast} \leq \Theta^{\circledast} \leq \Delta^{\circledast}+
\Upsilon^{\circledast}\leq \Om^{\circledast}+\Gamma^{\circledast},
\end{equation}
which proves \eqref{Eqmm}.
The fact that
if there exists a superpartition
$\Lambda$ such that
$\La^*= \Om^* + \Gamma^*$
and  $\La^{\circledast}=\Om^{\circledast} + \Gamma^{\circledast}$,
then
$t^{\La}_{\Om \Gamma}$ is non-zero is immediate since, as already observed,
the coefficient
of $e_{\Om'} \, e_{\Gamma'}$ in \eqref{eqapres}
is equal to $\pm 1$.

From the triangularity relation \eqref{Ptriangular},
the previous result implies
that $g^{\La}_{\Om \Gamma}=0$ unless  $\La^*
\leq \Om^* + \Gamma^*$ and
$\La^{\circledast} \leq \Om^{\circledast} + \Gamma^{\circledast}$.
 From Lemma~\ref{lemmaequivconj}, we have
\begin{equation} \label{lessi2}
g^{\La}_{\Om \Gamma}\neq 0\quad\text{ if and only if }
\quad g^{\La'}_{\Om' \Gamma'}\neq 0.
\end{equation}
But if
$g^{\La'}_{\Om' \Gamma'}\neq 0$ then
${\La^*}' \leq {\Om^*}' + {\Gamma^*}'$ and
${\La^{\circledast}}' \leq {\Om^{\circledast}}' + {\Gamma^{\circledast}}'$
which is equivalent to
$
\Om^* \cup \Gamma^* \leq \La^*$ and
$ \Om^{\circledast} \cup \Gamma^{\circledast} \leq \La^{\circledast}$.

As we have seen, if there exists a superpartition
$\Lambda$ such that
$\La^*= \Om^* + \Gamma^*$
and  $\La^{\circledast}=\Om^{\circledast} + \Gamma^{\circledast}$,
then
$t^{\La}_{\Om \Gamma}$ is non-zero.
By the triangularity \eqref{Ptriangular}, we have
$g^{\La}_{\Om \Gamma}\neq 0$ in those cases.  The final claim
follows from  (\ref{lessi2}).
\end{proof}

\subsection{Necessary conditions for the non-vanishing of coefficients
in the Pieri rule: horizontal and vertical strips}\label{sectstrips}
Let $n$ and $\tilde n$ refer respectively to the superpartitions
$(n)$ and $(n;)$, i.e.,   associated respectively to the following diagrams both containing $n$  squares:
\begin{equation}  n={\tableau[scY]{&&\bl\;\cdots&\bl&}}\quad\text{and}\quad
\tilde n={\tableau[scY]{&&\bl\;\cdots&\bl& &\bl\tcercle{}}}.
\end{equation}
We now obtain  necessary conditions for the
non-vanishing of the coefficients $g^\La_{\Omega,n}$ and
$g^\La_{\Omega,\tilde n}$.
These results specify -- without evaluating them explicitly -- the
coefficients that can appear in a Pieri-type rule for Jack polynomials
in superspace.

When no fermions are involved (in which case superpartitions
$\Lambda$ and $\Om$ are usual partitions $\la$ and $\mu$),
it is known that the coefficient $g^\la_{\mu,n}\neq 0$ if and only if
$\la/\mu$ is a horizontal $n$-strip.
The concept of horizontal or vertical strip can be easily generalized to superpartitions.

\begin{definition} \label{strips}
We say that $\La/\Om$
is a horizontal $n$-strip
if
$\La^*/\Om^*$  and
$\La^{\circledast}/\Om^{\circledast}$
are both horizontal $n$-strips.
Similarly, we say that $\La/\Om$
is a horizontal  $\tilde n$-strip
if
$\La^*/\Om^*$ is a horizontal $n$-strip and
$\La^{\circledast}/\Om^{\circledast}$
is a horizontal  $n+1$-strip. The definitions are similar for vertical strips.
\end{definition}
Consider for example, $\La=(4,1;2,1)$ and $\Om=(2,0;3,1)$.  Then, as
 illustrated in Figure \ref{hstrip}, $\La/\Om$ is a horizontal 3-strip, but it is not a vertical 3-strip.  Similarly, it is readily {seen}
from Figure \ref{vstrip} that $(3,0;2,1)/(2;2)$ is  a vertical $\tilde{2}$-strip.

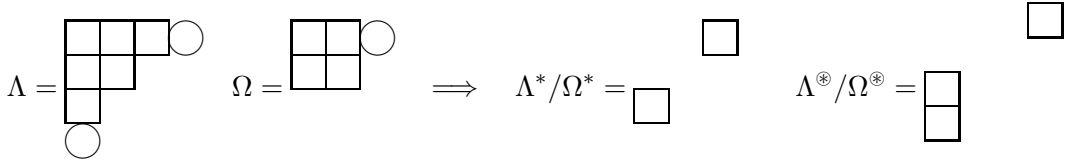
\begin{figure}[h]\caption{Horizontal $n$-strip}\label{hstrip}
{\begin{equation*}
\La={\tableau[scY]{&&&&\bl\tcercle{}\\&\\&\bl\tcercle{}\\&\bl}}\quad \Om={\tableau[scY]{&&\\&\bl\tcercle{}\\&\bl\\\bl\tcercle{}}} \quad\Longrightarrow\quad
 \La^*/\Om^*={\tableau[scY]{\bl&\bl&\bl&\\  \bl& \\ \bl & \bl  \\ &\bl  }}
\qquad
\La^\circledast/\Om^\circledast={\tableau[scY]{  \bl&\bl&\bl&&\\  \bl& \bl\\  \bl & \\ \bl }}
\end{equation*}}
\end{figure}

\begin{figure}[h]
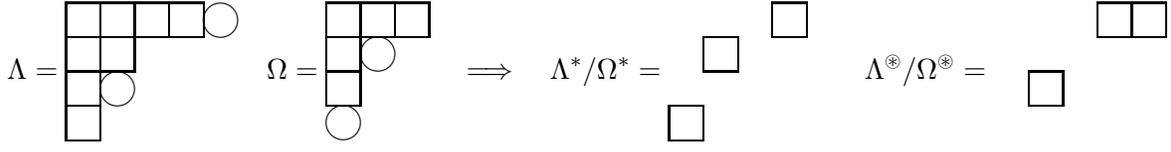
\caption{Vertical $\tilde n$-strip}\label{vstrip}
{\begin{equation*}
\La={\tableau[scY]{&&&\bl\tcercle{}\\&\\ \\  \bl\tcercle{} }}\quad \Om={\tableau[scY]{&&\bl\tcercle{}\\&\\  \bl \\\bl }} \quad\Longrightarrow\quad
\La^*/\Om^*={\tableau[scY]{\bl&\bl&\\  \bl& \bl \\ &\bl \\ \bl  }}
\qquad
\La^\circledast/\Om^\circledast={\tableau[scY]{\bl&\bl&\bl&\\  \bl& \bl\\   \\  \\ \bl}}
\end{equation*}}
\end{figure}

The proofs of the next two propositions rely on properties of non-symmetric
Jack polynomials.
As the latter are not used elsewhere and the demonstrations
are rather involved, they are
relegated to Appendix~\ref{appenB}. Note
that the equivalences in the statements follow from
Lemma~\ref{lemmaequivconj}.
\begin{proposition}\label{PropgI}
The coefficient $g^\La_{\Om,n}\neq0$
 only if $\La/\Om$ is a horizontal $n$-strip.  Equivalently,
the coefficient $g^\La_{\Om,1^n}\neq0$
 only if $\La/\Om$ is a vertical $n$-strip
\end{proposition}

\begin{proposition} \label{PropgII}
The coefficient  $g^\La_{\Om,\tilde n}\neq0$
only if
$\La/\Om$ is a horizontal $\tilde n$-strip. Equivalently,
the coefficient  $g^\La_{\Om,(0;1^n)}\neq0$
only if
$\La/\Om$ is a vertical $\tilde n$-strip
\end{proposition}

\begin{remark} Contrary to what occurs in the Pieri rule of Jack polynomials
\cite[Prop 5.3]{Stan}, the
{\it only if} in Propositions \ref{PropgI} and \ref{PropgII}
cannot be replaced by a
{\it if and only if}.  For example, if $\Om=(2;1)$, $\Lambda=(3;1)$ and $\tilde n= (1;)$ then it can be checked that $g^\La_{\Om,\tilde n}=0$ even though $\La/\Om$ is a horizontal
$\tilde 1$-strip.
\end{remark}
Recall that the diagram $\mu$ is contained in $\la$, denoted
$\mu\subseteq \la$, if $\mu_i\leq \la_i$ for all $i$.
For superpartitions
we define $\Om \subseteq \La$ as follows:
\begin{equation}\Om \subseteq\La\quad \text{if and only if} \quad\Om^* \subseteq \La^*\quad\text{and}\quad\Om^{\circledast} \subseteq \La^{\circledast} \, .
\end{equation}
{For instance, $(0;3,2)\subset(3,0;3,1)$ but $(2,1;3)\not\subseteq(3,0;3,1)$}.
Since the $e_{\Lambda}$'s form a multiplicative basis of
$\mathscr{R}$, the previous propositions have the following corollary.
\begin{corollary} \label{coro11}
We have that $g_{\Om \Gamma}^{\La}$ is zero unless $\Om \subseteq \La$ and $\Gamma \subseteq \La$.
\end{corollary}

\subsection{{Skew Jack polynomials in superspace}}\label{Sskewjack}
The skew Jack polynomial $P_{\La/\Om}$ is defined as the unique symmetric superfunction in $x$ and $\theta$  such that
\begin{equation}\label{defskew}
g^{\La}_{\Om \Gamma}=\LL \olw{P}_{\La/\Om}\,  | \,\orw{ P}_{\Gamma}\RR=\LL \olw{P}_\La   |  \orw{P}_\Om\orw{P}_\Gamma \RR  \,.
\end{equation}
Observe that this definition is equivalent to
\begin{equation}
P_{\Lambda/\Omega} = \sum_{\Gamma} \frac{g^{\Lambda}_{\Omega \Gamma}}{\|P_\Gamma\|^2}
P_{\Gamma} \, .
\end{equation}

\begin{lemma}\label{LemmaSkew}
\begin{equation}\sum_{\Om}\frac{1}{\|P_\Om\|^2}\orw{P}_\Om(x;\theta)\olw{P}_{\Om/\La}(y;\phi)=\sum_{\Om}{\frac{1}{\|P_\Om\|^2} }\orw{P}_\La(x;\theta)\orw{P}_\Om(x;\theta)\olw{P}_\Om(y;\phi).\end{equation}
\end{lemma}
\begin{proof}
 We denote the left-hand side and right-hand side of the equation by $F_\La(x;\theta|y;\phi)$ and $G_\La(x;\theta|y,\phi)$, respectively.
 By the linear independence and the orthogonality of the Jack polynomials, it is sufficient to prove that the following equation holds true for all superpartitions $\Gamma$ and $\Delta$:
 \begin{equation}
 \LL F_\La(x;\theta|y;\phi) |\olw{P}_\Ga(x;\theta)\, \orw{P}_\Delta(y;\phi)\RR=\LL  G_\La(x;\theta|y;\phi)
|\olw{P}_\Ga(x;\theta)\, \orw{P}_\Delta(y;\phi)\RR,\end{equation}
 where it is understood that two independent scalar products are taken, first with respect to the indeterminates $x$ and $\theta$, and then with respect to  $y$ and $\theta$.
 Thus, the left-hand side (LHS) is
 \begin{align}
 \text{LHS}
&= \sum_\Om\frac1{\|P_\Om\|^2}\LL\orw{P}_\Om(x;\ta)|\olw{P}_\Gamma(x;\ta)\RR (-1)^{\underline{\overline{\Ga}}(\underline{\overline{\Om}}-\underline{\overline{\La}})}\LL\olw{P}_{\Om/\La}(y;\phi)|\orw{P}_\Delta(y;\phi)\RR\nonumber\\
 &=(-1)^{\underline{\overline{\Ga}}(\underline{\overline{\Ga}}-\underline{\overline{\La}})}\LL\olw{P}_{\Ga/\La}(y;\phi)|\orw{P}_\Delta(y;\phi)\RR\nonumber\\
& =(-1)^{\underline{\overline{\Ga}}\cdot\underline{\overline{\Delta}}}\,\LL\olw{P}_{\Ga}(y;\phi)|\orw{P}_\La(y;\phi)\orw{P}_\Delta(y;\phi)\RR.
 \end{align}
The sign in the first equality comes from the reordering of the two terms in the product $\olw{P}_{\Om/\La}\olw{P}_\Gamma$. To get the second line, we used the obvious equality $\LL\orw{P}_\Om|\olw{P}_\Gamma\RR=\LL\olw{P}_\Om|\orw{P}_\Gamma\RR$. Finally, since the scalar product in the third line is non-zero only if $\underline{\overline{\Ga}}=\underline{\overline{\La}}+\underline{\overline{\Delta}}$, this is used to simplify the phase factor.
Similarly for the right-hand side, we have
\begin{align}
 \text{RHS}
&= \sum_\Om\frac1{\|P_\Om\|^2}\LL\orw{P}_\La(x;\ta)\orw{P}_\Om(x;\ta)|\olw{P}_\Gamma(x;\ta)\RR
(-1)^{\underline{\overline{\Ga}}\cdot\underline{\overline{\Om}}}
\LL\olw{P}_{\Om}(y;\phi)|\orw{P}_\Delta(y;\phi)\RR\nonumber\\
 &=(-1)^{\underline{\overline{\Ga}}\cdot\underline{\overline{\Delta}}}\, \LL\orw{P}_{\La}(x;\ta)\orw{P}_\Delta(x;\ta)|\olw{P}_\Ga(x;\ta)\RR,
 \end{align}
 which proves the lemma.
\end{proof}

\begin{proposition}\label{PropSkew}
 Let $(x,y;\theta,\phi)$ denote the ordered set
$(x_1,x_2,\ldots,y_1,y_2,\ldots;\theta_1,\theta_2,\ldots,\phi_1,\phi_2,\ldots).$
Then,
we have
\begin{equation}\label{EqSkew}\begin{split}
{P}_\Gamma(x,y;\theta,\phi)&=\sum_\La\frac{1}{\|P_\La\|^2}\,{P}_\La(x;\theta)
{P}_{\Gamma/\La}(y;\phi).\end{split}\end{equation}
Moreover, the following generalization holds
\begin{equation}\label{EqSkew2}
{P}_{\Gamma/\Om}(x,y;\theta,\phi)=\sum_\La\frac{1}{\|P_\La\|^2}\, {P}_{\La/\Om}(x;\theta) {P}_{\Ga/\La}(y;\phi),
\end{equation}
reducing to the previous identity when $\Om$ is the empty superpartition.
\end{proposition}
\begin{proof}
Let us first point out that throughout the proof, we use of the obvious identity $(-1)^{\underline{\overline{\Ga}}\cdot \underline{\overline{\Ga}}}=(-1)^{\underline{\overline{\Ga}}}$, which is true for any superpartition
since  $n^2=n \mod 2$ for any integer $n$.

Now let $(z;\tau)=(z_1,z_2,\ldots;\tau_1,\tau_2,\ldots)$.  The use of the Cauchy formula allows us to write
\begin{equation}
\begin{split}\sum_{\Ga}\frac{1}{\|P_\Ga\|^2}\olw{P_\Ga}(x,y;\theta,\phi)\orw{P_\Ga}(z;\tau)&=
\prod_{i}\prod_{j}\prod_{k}(1-x_jz_k-\theta_j\tau_k)^{-1/\alpha}(1-y_iz_k-\phi_i\tau_k)^{-1/\alpha}\\
&=\sum_{\La,\Gamma}\frac{1}{\|P_\La\|^2 \|P_\Ga\|^2}\olw{P_\La}(x;\theta)\orw{P_\La}(z;\tau)\olw{P_\Gamma}(y;\phi)\orw{P_\Gamma}(z;\tau)\\
&=\sum_{\La}\frac{1}{\|P_\La\|^2}\olw{P_\La}(x;\theta) \sum_\Ga\frac1{ \|P_\Ga\|^2}\orw{P_\La}(z;\tau)
\orw{P_\Gamma}(z;\tau)\olw{P_\Gamma}(y;\phi)\,(-1)^{\underline{\overline{\Ga}}\cdot \underline{\overline{\Ga}}},
\end{split}\end{equation}
where in the last step, the last two terms have been interchanged.
 We then use the identity
 \begin{equation}(-1)^{\underline{\overline{\Ga}}\cdot \underline{\overline{\Ga}}}\,\olw{P_\Gamma}(y;\phi)=\olw{P_\Gamma}(y;-\phi)\end{equation} and  Lemma \ref{LemmaSkew} to get
\begin{equation}
\begin{split}
 \sum_{\Ga}\frac{1}{\|P_\Ga\|^2}\olw{P_\Ga}(x,y;\theta,\phi)\orw{P_\Ga}(z;\tau)
 &=\sum_{\La}\frac{1}{\|P_\La\|^2}\olw{P}_\La(x;\theta) \sum_\Ga\frac1{ \|P_\Ga\|^2}\orw{P_\Ga}(z;\tau)
\olw{P}_{\Gamma/\La}(y;-\phi)
\\&=\sum_{\La}\frac{1}{\|P_\La\|^2}\olw{P}_\La(x;\theta) \sum_\Ga\frac1{ \|P_\Ga\|^2}
\olw{P}_{\Gamma/\La}(y;\phi)\orw{P_\Ga}(z;\tau) \, (-1)^{\underline{\overline{\La}}\cdot( \underline{\overline{\Ga}}- \underline{\overline{\La}})},
\end{split}
\end{equation}
where the last line has been simplified thanks to $(-1)^{\underline{\overline{\Ga}}-\underline{\overline{\La}}
+\underline{\overline{\Ga}}\cdot( \underline{\overline{\Ga}}- \underline{\overline{\La}})}=(-1)^{\underline{\overline{\La}}\cdot( \underline{\overline{\Ga}}- \underline{\overline{\La}})}$.
Again, the linear independence
of the Jack superpolynomials allow us to equate the coefficients of $\orw{P_\Ga}(z;\tau) $ on both sides.
We {finally} permute $\olw{P}_\La(x;\theta)$ with $\olw{P}_{\Gamma/\La}(y;\phi)$ in order to cancel out the factor  $(-1)^{\underline{\overline{\La}}\cdot( \underline{\overline{\Ga}}- \underline{\overline{\La}})}$ and get
\begin{equation}\label{EqSkew1}\begin{split}
\olw{P_\Ga}(x,y;\theta,\phi)&=\sum_\La\frac{1}{\|P_\La\|^2}\,\olw{P}_{\Gamma/\La}(y;\phi)\olw{P_\La}(x;\theta).\end{split}\end{equation}
which is equivalent to \eqref{EqSkew}.

To prove the second part,  we first observe that  the symmetry property of $\olw{P}_\La$ allows us to interchange the variables $(x;\theta)$ and $(y;\phi)$ in   \eqref{EqSkew1}, so that
\begin{equation}\label{EqSkewa}\begin{split}
\olw{P}_{{\Gamma}}(x,y;\theta,\phi)&=\sum_\La\frac{1}{\|P_\La\|^2}\,\olw{P}_{\Gamma/\La}(x;\theta) \olw{P_\La}(y;\phi).\end{split}\end{equation}
 Writing  \eqref{EqSkewa} in terms of three sets of variables yields
\begin{equation}
\begin{split}
\olw{P_\Gamma}(x,y,z;\theta,\phi,\tau)&=\sum_\La\frac{1}{\|P_\La\|^2}\olw{P}_{\Gamma/\La}(x;\theta)\olw{P}_\La(y,z;\phi,\tau).
\end{split}
\end{equation}
Using \eqref{EqSkewa} to expand $\olw{P}_\La(y,z;\phi,\tau)$,
we obtain
\begin{equation}
\olw{P_\Gamma}(x,y,z;\theta,\phi,\tau)=\sum_{\La,\Om}\frac{1}{\|P_\La\|^2\|P_\Om\|^2}\olw{P}_{\Gamma/\La}(x;\theta)\olw{P}_{\La/\Om}(y;\phi)\olw{P_\Om}(z;\tau).\end{equation}
However, there is another way to write   \eqref{EqSkewa} in terms of three sets of variables:
\begin{equation}
\olw{P_\Gamma}(x,y,z;\theta,\phi,\tau)=\sum_\Om\frac{1}{\|P_\Om\|^2}\olw{P}_{\Gamma/\Om}(x,y;\theta,\phi)\olw{P}_\Om(z;\tau).
\end{equation}
Equating the coefficients of $\olw{P_\Om}(z;\tau)$ in the last two expressions of $\olw{P_\Gamma}(x,y,z;\theta,\phi,\tau)$ gives
\begin{equation}
\olw{P}_{\Gamma/\Om}(x,y;\theta,\phi)=\sum_\La\frac{1}{\|P_\La\|^2}\,\olw{P}_{\Gamma/\La}(x;\theta)\olw{P}_{\La/\Om}(y;\phi).
\end{equation}
Finally,  \eqref{EqSkew2} is established
by interchanging $(x;\theta)$ and $(y;\phi)$, and by reversing the ordering of all the Grassmann variables.
\end{proof}

\section{Decomposition of Jack superpolynomials}\label{facpro}

Our main results, presented in Sections~\ref{spe1} and \ref{spe2},
 rely in an essential way on certain column-wise decomposition
properties of the Jack polynomials, presented in
Section~\ref{codec}, that generalize known properties of Jack polynomials.
The analogous row-wise decompositions, worked out in Section~\ref{rodec},
are given for completeness.

\subsection{Column operations}\label{codec}

If the first column of the diagram of $\La$ does not contain a circle,
we introduce the ``column-removal'' operation $\C$ defined such that
$\C \La$ is the
superpartition
whose diagram is obtained by removing the
first column of the diagram of $\La$ (the
operation is illustrated in Fig.~\ref{FigC}).

If the first column of the diagram of
$\La $ contains a circle, we define the ``circle-removal'' operation  $\widetilde{\C}$ such that the diagram of $\widetilde{\C} \La$ is obtained from that of $\La$ by removing the circle in the first column of the diagram of $\La$
(also illustrated in Fig.~\ref{FigC}).

\begin{figure}[h]\caption{Operators $\C$ and $\widetilde{\C}$}\label{FigC}
 \begin{equation*}\mathcal{C} \,:\quad {\tableau[scY]{&&&\\&&\bl\tcercle{}\\&\\&\bl\tcercle{}}}  \longmapsto \quad {\tableau[scY]{&&\\&\bl\tcercle{}\\ \\\bl\tcercle{}}}
\qquad \widetilde{\mathcal{C}} \,:\quad {\tableau[scY]{&&&\bl\tcercle{}\\&\\&\bl\tcercle{}\\\bl\tcercle{}}}  \longmapsto \quad {\tableau[scY]{&&&\bl\tcercle{}\\&\\&\bl\tcercle{}\\\bl }}
 \end{equation*}
\end{figure}

\begin{proposition}\label{PropFactoI}Let $\La$ be a superpartition having no parts equal to zero,
i.e., $\ell(\La)=\ell(\La^*)=\ell$.  Then
\begin{equation}
P_\La(x_1,\ldots,x_\ell;\theta_1,\ldots,\theta_\ell)=  x_1\cdots x_\ell \,
 P_{\C\La}(x_1,\ldots,x_\ell;\theta_1,\ldots,\theta_\ell) .\end{equation}
\end{proposition}
\begin{proof}
To simplify the notation, we will assume throughout the proof that
the polynomials are polynomials in the variables
$(x_1,\ldots,x_\ell;\theta_1,\ldots,\theta_\ell)$.  Since
in that case
$m_{1^\ell}=x_1\cdots x_\ell$,
we have to show that $ m_{1^\ell}\, P_{\C\La}=P_\La$.
According to Theorem~\ref{TheoEigenJack}, this amounts to prove
that $m_{1^\ell}\, P_{\C\La}$ is:  $(i)$
triangular with leading term $m_\La$ and, $(ii)$
an eigenfunction of  $D$ and $\Delta$ with eigenvalues
$\varepsilon_\La$ and $\epsilon_\La$ respectively.

From the definition of the monomial symmetric functions, we immediately get
\begin{equation}  m_{1^\ell}\, m_{\C\Gamma}=m_\Gamma,\end{equation}
for every $\Gamma$ {such that}  $\ell(\Gamma)=\ell(\Gamma^*)=\ell$.
From the diagrammatic representation of superpartitions,  it is also obvious that if $\ell(\Gamma)=\ell(\Gamma^*)=\ell$ then
\begin{equation}
\La\geq\Gamma\qquad\Longleftrightarrow\qquad \C\La\geq \C\Gamma.\end{equation}
Hence, we obtain
\begin{equation}
\begin{split}
m_{1^\ell}P_{\C\La}=m_{1^\ell}\big(m_{\C\La} +\sum_{\Om<\C\La}c_{\C\La,\Om}(\alpha)\, m_\Om\big)= & \, m_{1^\ell}\big(m_{\C\La} +\sum_{\C \Gamma<\C\La}c_{\C\La,\C \Gamma}
(\alpha) \, m_{\C \Gamma}\big) \\
= &  \,
m_\La+\sum_{\Gamma<\La}c_{\C\La,\C\Gamma}(\alpha) \, m_{\Gamma},
\end{split}
\end{equation}
which proves the triangularity of $m_{1^\ell}P_{\C\La}$.

We now compute the action of $D$ on $m_{1^\ell}P_{{\C \La}}$:
\begin{equation}
\begin{split}D(m_{1^\ell}P_{{\C \La}})&=D(m_{1^\ell})\,P_{\C \La}+m_{1^\ell}
D(P_{\C \La})+
\alpha\sum_ix_i\partial_{x_i}(m_{1^\ell})x_i\partial_{x_i}(P_{\C \La})\\
&=(\varepsilon_{1^\ell}(\alpha)+
\varepsilon_{\C \La}(\alpha)+ \alpha|{\C \La} |\, )m_{1^\ell}P_{\C \La}.
\end{split}\end{equation}
Using
\begin{equation}b(\C\La)=b(\La)-\frac{\ell(\ell-1)}{2},\quad b((\C\La)')=b(\La')-|\C\La|\quad
{\rm and} \quad \varepsilon_{1^\ell}(\alpha)=-\frac{\ell(\ell-1)}{2},
\end{equation}
which implies that $\varepsilon_{1^\ell}(\alpha)+\varepsilon_{\C\La}(\alpha)
+\alpha|\C\La| = \varepsilon_\La
(\alpha)$, we get
\begin{equation}
D(m_{1^\ell}P_{\C\La})=\varepsilon_\La(\alpha)m_{1^\ell}P_{\C \La}.\end{equation}
Similarly, the action of $\Delta$ on $m_{1^\ell}P_{\C \La}$ gives
\begin{equation}
\begin{split}
\Delta(m_{1^\ell}P_{\C \La})
&=\Delta(m_{1^\ell})\,P_{\C \La}+m_{1^\ell}\Delta(P_{ \C \La})+\alpha\sum_ix_i\partial_{x_i}(m_{1^\ell})\theta_i\partial_{\theta_i}(P_{\C \La}) \\
&=\left(0+\epsilon_{\C \La}(\alpha)+\alpha \,  m \right)m_{1^\ell}P_{\C \La}.
\end{split}
\end{equation}
Using $|(\C \La)^{a}|=|\La^{a}|-m$ and $|{(\C \La)'}^{a}|=|{\La'}^{a}|$, we are led to
\begin{equation}\Delta(m_{1^\ell}P_{\C \La})=\epsilon_\La(\alpha) \, m_{1^\ell}
P_{\C \La}.\end{equation}
We have established that $m_{1^\ell}P_{\C\La}$ has the right triangularity
property and satisfies the required eigenvalue problems.  We can thus
conclude that $m_{1^\ell}P_{\C\La}=P_{\La}$.
\end{proof}

\begin{proposition}\label{PropFactoII}Let $\La$ be a superpartition such that $\La_m=0$,
{i.e.,} $\ell(\La)=\ell(\La^*)+1$. Then
\begin{equation}
(-1)^{m-1}
\Bigl[ \partial_{\theta_\ell}\,P_\La(x_1,\ldots,x_\ell;\theta_1,\ldots,\theta_\ell)
\Bigr]_{x_\ell=0}
=
P_{\widetilde{\C}\La}(x_1,\ldots,x_{\ell-1};\theta_1,\ldots,\theta_{\ell-1}).
\end{equation}
\end{proposition}
\begin{proof}Let
$(x;\theta)=(x_1,\ldots,x_\ell;\theta_1,\ldots,\theta_\ell)$ and $(x_-;\theta_-) = (x_1,\ldots,x_{\ell-1};\theta_1,\ldots,\theta_{\ell-1})$.
As in Proposition~\ref{PropFactoI}, we will prove that
$\partial_{\theta_\ell} P_\La(x;\theta)$ evaluated at ${x_\ell=0}$ has the
right triangularity and satisfies the required eigenvalue problems
(see Theorem~\ref{TheoEigenJack}).

We have, for any superpartition $\Om$ with $\ell(\Om^{\circledast})
= \ell$,
\begin{equation}\label{Eqm}
(-1)^{m-1}
\Bigl[m_\Om(x;\theta)\Bigr]_{x_\ell=0}=
\begin{cases}\theta_{\ell} \, m_{\tilde{\C}\Om}(x_-;\theta_-) & {\rm if~}\Om_m=0\\
0 & {\rm if~}\Om_m>0. \end{cases}
\end{equation}
Moreover, one readily shows that whenever $\La_m=\Om_m=0$, we have
\begin{equation}
\La\geq\Om\qquad\Longleftrightarrow\qquad \tilde{\C}\La\geq \tilde{\C}\Om
.\end{equation}
Therefore, from the expansion \eqref{Ptriangular}
of $P_{\La}(x,\theta)$ in terms of monomials,
we immediately get
\begin{equation} \label{EqTriangzero}
\begin{split}
(-1)^{m-1}\Bigl[\partial_{\theta_\ell}P_\La(x;\theta)\Bigr]_{x_\ell=0}
& =m_{\tilde{\C}\La}(x_-;\theta_-)+\sum_{\substack{\Om<\La\\\Om_m=0}}c_{\La,\Om}(\alpha)\,m_{\tilde{\C}\Om}(x_-;\theta_-) \\
& =m_{\tilde{\C}\La}(x_-;\theta_-)+
\sum_{\substack{\tilde{\C} \Om< \tilde{\C} \La\\\Om_m=0}}c_{\La,\Om}(\alpha)\,m_{\tilde{\C}\Om}(x_-;\theta_-),
\end{split}
\end{equation}
which gives the desired triangularity. Observe that we have used the fact
that if $\Om \leq \La$ then $\ell(\Om^{\circledast})
\geq \ell(\La^{\circledast})=\ell$.

Let $D_{\ell}$ and $D_{\ell-1}$
stand for the operator $D$ in the variables
$(x,\theta)$ and $(x_-,\theta_-)$ respectively, and
similarly for  $\Delta_{\ell}$ and $\Delta_{\ell-1}$.
For any superpolynomial $f(x,\theta)$, straightforward calculations yield
\begin{equation}
\Bigl[ \partial_{\theta_\ell} D_{\ell} f(x,\theta)  \Bigr]_{x_\ell=0} =
\Bigl[ \partial_{\theta_\ell} D_{\ell-1} f(x,\theta)  \Bigr]_{x_\ell=0}
= D_{\ell-1}\Bigl[ \partial_{\theta_\ell}  f(x,\theta)  \Bigr]_{x_\ell=0} .
\end{equation}
Therefore
\begin{equation} \label{EqEigenD}
D_{\ell-1}\Bigl[\partial_{\theta_\ell}P_\La(x,\theta)\Bigr]_{x_\ell=0}=
\Bigl[\partial_{\theta_\ell} D_{\ell} P_\La(x,\theta)\Bigr]_{x_\ell=0}
=\varepsilon_{\tilde{\C}\La}(\alpha)\Bigl[\partial_{\theta_\ell}
P_\La(x,\theta) \Bigr]_{x_\ell=0},
\end{equation}
since $\varepsilon_\La(\alpha)
=\varepsilon_{\tilde{\C}\La}(\alpha)$.

The second eigenvalue problem is somewhat more involved.   We have,
for any superpartition $\Om$ with $\ell(\Om^{\circledast})= \ell$,
\begin{equation}
\begin{split}
\Bigl[\partial_{\theta_\ell}(\Delta_{\ell-1}-\Delta_{\ell})\,
m_\Om(x;\theta)\Bigr]_{x_\ell=0} & = \partial_{\theta_\ell} \Bigl[(\Delta_{\ell-1}-\Delta_{\ell}) \Bigr]_{x_\ell=0} \Bigl[ m_\Om(x;\theta)\Bigr]_{x_\ell=0} \\
& = \partial_{\theta_\ell} \sum_{j=1}^{\ell-1} \theta_\ell
(\partial_{\theta_\ell} -\partial_{\theta_j})
\Bigl[ m_\Om(x;\theta)\Bigr]_{x_\ell=0} \\
&=
\left \{ \begin{array}{ll}
(\ell-1)(-1)^{m-1}m_{\tilde{\C}\Om}(x_-;\theta_-) & {\rm if~} \Omega_m=0 \\
0 & {\rm if~} \Omega_m> 0
\end{array} \right. \\
& = (\ell-1) \Bigl[\partial_{\theta_\ell}\,
m_\Om(x;\theta)\Bigr]_{x_\ell=0} ,
\end{split} \end{equation}
 where we have used \eqref{Eqm} in the next to last step.
This leads to
\begin{equation} \label{EqEigenDelta}
\begin{split}
\Delta_{\ell-1}  \Bigl[\partial_{\theta_\ell}P_\La(x,\theta)\Bigr]_{x_\ell=0}
  & =
\Bigl[\partial_{\theta_\ell} (\Delta_{\ell}+(\Delta_{\ell-1} - \Delta_{\ell}))
P_\La(x,\theta)\Bigr]_{x_\ell=0} \\
& = (\epsilon_{\La}(\alpha)+\ell-1) \Bigl[\partial_{\theta_\ell}
P_\La(x,\theta)\Bigr]_{x_\ell=0} \\
& = \epsilon_{\tilde{\C}\La}(\alpha) \Bigl[\partial_{\theta_\ell}
P_\La(x,\theta)\Bigr]_{x_\ell=0},
\end{split}
\end{equation}
since $|(\tilde{\C}\La)^{a}|= |\La^{a}|$,
$|({\tilde{\C}\La)'}^{a}| = |\La^{a}|-\La'_1$, and $\La'_1=\ell-1$.

Using Eqs \eqref{EqTriangzero},\eqref{EqEigenD} and \eqref{EqEigenDelta},
we conclude from Theorem \ref{TheoEigenJack} that
\begin{equation}(-1)^{m-1}\Bigl[\partial_{\theta_\ell}P_\La(x;\theta)\Bigr]_{x_\ell=0}=P_{ \tilde{\C}\La}(x_-;\theta_-) \, .\end{equation}
\end{proof}

\subsection{Row operations}\label{rodec}
Similarly to the column case, we can introduce two row operations whose
actions on diagrams is illustrated in Fig. \ref{FigRtilde}.
The following two propositions show how the polynomial $P_\La$ can be
row-wise
deconstructed  (at the expense of losing the first variable).
\begin{figure}[h]\caption{Operators $\R$ and $\widetilde{\R}$}\label{FigRtilde}
\begin{equation*}\mathcal{R} \,:\quad {\tableau[scY]{&&&\\&&\bl\tcercle{}\\&\\&\bl\tcercle{}}}  \longmapsto \quad {\tableau[scY]{ &&\bl\tcercle{}\\& \\&\bl\tcercle{}\\ \bl }}
\qquad \widetilde{\mathcal{R}} \,:\quad {\tableau[scY]{&&&\bl\tcercle{}\\&\\&\bl\tcercle{}\\&\bl }}  \longmapsto \quad {\tableau[scY]{&&&\bl \\&\\&\bl\tcercle{}\\&\bl }}
 \end{equation*}
\end{figure}

\begin{proposition}\label{PropFactoIdual}
Let $(x_-;\theta_-)=(x_2,x_3,\dots;\theta_2, \theta_3,\dots)$.  Let also $\La$ be a superpartition whose fermionic degree is
$m$. If the first row of the diagram of $\Lambda$ is bosonic
(that is, $\La^*_1= \La^\circledast_1=k$),
then
\begin{equation}\coeff{x_1^k} P_\La(x;\theta) = P_{\R\La}(x_-;\theta_-).
\end{equation}
\end{proposition}
\begin{proof}  We have from Proposition~\ref{PropSkew} that
\begin{equation} \label{eq48}
{P}_{\La}(x,\theta) = \sum_{\Om} \frac{1}{\|P_\Om\|^2}
{P}_{\Om}(x_1,\theta_1) \, {P}_{\La/\Om}(x_-,\theta_-).
\end{equation}
Using
\begin{equation} \label{eq49}
{P}_{\Om}(x_1,\theta_1) =
\left\{\begin{array}{ll}
x_1^r  & {\rm if~} \Omega=r \\
\theta_1 x_1^r  & {\rm if~} \Omega=\tilde r \\
0 & {\rm otherwise},
\end{array}
\right.
\end{equation}
and the fact that
$P_{\La/\Om}=0$ unless $\Om \subseteq \La$ from Corollary~\ref{coro11}
(that is, unless $\Omega_1^* \leq k$ and $\Omega_1^{\circledast} \leq k)$,
we obtain
\begin{equation}
{P}_{\La}(x,\theta) = \sum_{r= 0}^{k} \frac{1}{\|P_{r}\|^2}  x_1^r
{P}_{\La/r}(x_-,\theta_-) +
 \sum_{r=0}^{k-1} \frac{1}{\|P_{\tilde r}\|^2} \theta_1 x_1^r
{P}_{\La/ \tilde r}(x_-,\theta_-) \, ,
\end{equation}
which immediately gives
\begin{equation}
\coeff{x_1^k} P_\La(x;\theta) \propto P_{\La/k} (x_-,\theta_-) .
\end{equation}
Moreover, we have from \eqref{defskew} that
\begin{equation}
{P}_{\La/k} (x_-,\theta_-)  = \sum_{\Om} \frac{g_{\Om, k}^{\La}}{\|P_\Om\|^2}
\, {P}_{\Om}(x_-,\theta_-),
\end{equation}
where we recall from Proposition \ref{PropgI} that
 $g_{\Om, k}^{\La}\ne 0$ only if
  $\La/\Om$ is a
horizontal $k$-strip.
Now, the only superpartition $\Om$ such that $\La/\Om$ is a horizontal
$k$-strip is
$\R \La$, and therefore
\begin{equation}
\coeff{x_1^k} P_\La(x;\theta) \propto P_{\R \La} (x_-,\theta_-).
\end{equation}
Finally, since
\begin{equation}
\coeff{x_1^k}  \, m_\La(x;\theta)  = m_{\R \La} (x_-,\theta_-),
\end{equation}
we have that $\coeff{x_1^k} \,  P_\La(x;\theta)$ is monic and the proposition follows.
\end{proof}

\begin{proposition}\label{PropFactoIIdual}Let $(x_-;\theta_-)
=(x_2,x_3,\dots;\theta_2, \theta_3,\dots)$.
Let also $\La$ be a superpartition whose fermionic degree
is $m$. If the first row of the diagram of $\La$ is fermionic
(that is, $\La_1^*=\La_1^\circledast-1=k$),
then
\begin{equation}\coeff{x_1^k}\, \partial_{\theta_1} \, P_\La(x;\theta)
= P_{\R\widetilde{\R}\La}(x_-;\theta_-).
\end{equation}
\end{proposition}
\begin{proof} The proof is essentially the same as that of Proposition~\ref{PropFactoIdual}.
Using Eqs. \eqref{eq48} and \eqref{eq49}, and the fact that $P_{\La/\Om}=0$ unless $\Om \subseteq \La$
(that is, unless $\Omega_1^* \leq k$ and $\Omega_1^{\circledast} \leq k+1)$,
we obtain
\begin{equation}
{P}_{\La}(x,\theta) = \sum_{r=0}^{k} \frac{1}{\|P_{r}\|^2}
 x_1^r {P}_{\La/r}(x_-,\theta_-) +
 \sum_{r=0}^{k}  \frac{1}{\|P_{\tilde r}\|^2} \theta_1 x_1^r
{P}_{\La/\tilde r}(x_-,\theta_-)\, ,
\end{equation}
which immediately gives
\begin{equation}
\coeff{x_1^k} \, \partial_{ \theta_1} P_\La(x;\theta) \propto P_{\La/\tilde k} (x_-,\theta_-).
\end{equation}
Now
\begin{equation}
P_{\La/\tilde k} (x_-,\theta_-)  = \sum_{\Om} \frac{g_{\Om, \tilde k}^{\La}}{\|P_\Om\|^2}
\, P_{\Om}(x_-,\theta_-)
\end{equation}
is such that $g_{\Om ,\tilde k}^{\La}=0$ unless $\La/\Om$ is a horizontal
$\tilde k$-strip, which readily yields
\begin{equation}
\coeff{x_1^k} \, \partial_{\theta_1} P_\La(x;\theta) \propto
P_{\R \tilde{\R} \La} (x_-,\theta_-) .
\end{equation}
Finally, since
\begin{equation}
\coeff{x_1^k}  \, \partial_{\theta_1}  m_\La(x;\theta)  = m_{\R \tilde {\R}\La} (x_-,\theta_-),
\end{equation}
we have that $\coeff{x_1^k} \, \partial_{\theta_1} P_\La(x;\theta)$ is monic and the proposition follows.

\end{proof}

\section{Evaluation formulas}\label{speS}
We now come to our first main results: the derivation of evaluation
formulas for the Jack superpolynomials.   In what follows, it will prove  convenient
 to work with a
different normalization of the Jack polynomials in superspace.
Let $\La_{{\rm min}}$ be the lowest superpartition  of degree $(n|m)$
in the dominance ordering, {namely}:
\begin{equation}\label{minide}
\Lambda_{\mathrm{min}}:=(\delta_m\,;\,
1^{\ell_{n,m}}\,) \, ,\end{equation} where
\begin{equation}\label{delnm}
\ell_{n,m}:=n-|\delta_m| \aand \delta_m:=(m-1,m-2,\ldots,0)
\, .\end{equation}
Let also
$c_\La^{\mathrm{min}}(\alpha)$
stand for the coefficient of
$\ell_{n,m}! \,m_{\Lambda_{\mathrm{min}}}$ in the monomial expansion of
$P_\La$.

\begin{definition}\label{defJackJ}
We define the non-monic Jack symmetric function in superspace as
\begin{equation}\label{eqnormalJ}
J_\La:= v_\La(\alpha) P_\La=\frac{1}{c_\La^{\mathrm{min}}(\alpha)} P_\La.\end{equation}
\end{definition}

 This normalization, which is such that the coefficient of $m_{\Lambda_{\mathrm{min}}}$
in $J_{\La}$ is $\ell_{n,m}!$,
reduces to the integral form of the Jack polynomials  \cite{Stan}
when $m=0$.
We define the expansion coefficients of $J_\La$ as
\begin{equation}
J_\La=\sum_{\Om\leq\La}v_{\La\Om}(\alpha)\,m_\Om,\end{equation}
 with the identification:
 \begin{equation}v_\La (\alpha)\equiv v_{\La\La} (\alpha).\end{equation}

\begin{remark} It has been conjectured in {\cite[Conj. 33]{DLMadv}} that the coefficients $v_{\La\Om}(\alpha)$ belong to $\mathbb Z[\alpha]$.
This conjecture is still open.
\end{remark}

\subsection{First evaluation formula}\label{spe1}
Let $F(x;\theta)$ be a polynomial in superspace of fermionic degree $m$.
The evaluation of  $F(x;\theta)$ is defined as
\begin{equation}\label{EqDefSpecialI}
\Sp_{N,m}[F(x;\ta)] := \left[
\frac{\partial_{\theta_m} \cdots \partial_{\theta_1}
F(x;\theta)}{V_m(x)}\right]_{\substack{x_1=\ldots=x_N=1}},\end{equation}
where $m\leq N$ and where
\begin{equation}
V_m(x)={\prod_{1\leq i<j\leq m}(x_i-x_j)}
\end{equation}
is the Vandermonde determinant in the variables $x_1,\dots,x_m$.
Note that it is
understood that $V_m(x)=1$ when $m=0$ or $1$, and thus
this evaluation reduces to the standard evaluation at $x_1=\cdots=x_N=1$
when $m=0$.

It will prove useful to reexpress the division by the Vandermonde determinant $V_m(x)$ in the evaluation \eqref{EqDefSpecialI}
as a differentiation with respect to the operator
\begin{equation}
\partial_x^{\delta} :=  \left[ \prod_{i=1}^{m-1} \frac{1}{(m-i)!} \right]
\partial_{x_1}^{m-1} \partial_{x_2}^{m-2} \cdots \partial_{x_{m-1}}.
\end{equation}
\begin{lemma} Let $F(x;\theta)$ be a polynomial in superspace of fermionic degree $m$.  Then
\begin{equation}\label{Spasder}
\Sp_{N,m}[F(x;\ta)]=\left[
 \partial_x^{\delta}
\partial_{\theta_m} \cdots \partial_{\theta_1}
F(x;\theta)\right]_{\substack{x_1=\ldots=x_N=1}}\end{equation}
\end{lemma}
\begin{proof}
If $f(x_1,\dots,x_N)$ is a polynomial antisymmetric in the variables
$x_1,\dots,x_m$, then
\begin{equation}
f(x_1,\dots,x_N)=
g(x_1,\dots,x_N) V_m(x)
\end{equation}
for some polynomial $g(x_1,\dots,x_N)$.  This implies that
\begin{equation}
\left[\frac{f(x_1,\dots,x_N)}{V(x)}
\right]_{x_1=\cdots=x_N=1}=\left[ g(x_1,\dots,x_N)
\right]_{x_1=\cdots=x_N=1},
\end{equation}
 from which we have
\begin{align}\label{divV}
 \left[ \frac{f(x_1,\dots,x_N)}{V_m(x)} \right]_{x_1=\cdots=x_N=1}
 & = \left[ g(x_1,\dots,x_N) \partial_x^{\delta} \prod_{1\leq i <j \leq m}(x_i-x_j)
\right]_{x_1=\cdots=x_N=1} \nonumber\\
 & = \left[ \partial_x^{\delta}
\Big[g(x_1,\dots,x_N)  \prod_{1\leq i <j \leq m}(x_i-x_j)\Big]
\right]_{x_1=\cdots=x_N=1} \nonumber\\
& = \left[ \partial_x^{\delta}
f(x_1,\dots,x_N)\right]_{x_1=\cdots=x_N=1}.
\end{align}
If $F(x,\theta)$ is a polynomial in superspace of fermionic degree $m$,
then $\partial_{\theta_m} \cdots \partial_{\theta_1} F(x,\theta)$ is a polynomial
in $x_1,\dots,x_N$ antisymmetric in $x_1,\dots,x_m$.  This proves the lemma.
\end{proof}

\begin{lemma}\label{LemmaSpecialIm}
Let $\La=(\La^a;\La^s)$ be a superpartition of length $\ell\leq N$ and of fermionic degree $m$. Let also
$\gamma=\La^a-\delta_m$ where $\delta_m=(m-1,m-2,\ldots,0)$.
Then
\begin{equation}\Sp_{N,m}\,[m_\La]=\frac{(N-\ell+1)_{\ell-m}}{n_{\La^s}!}\prod_{(i,j)\in\gamma}\frac{m+j-i}{l_{\gamma}(i,j)+a_{\gamma}(i,j)+1} \, ,\end{equation}
where $(a)_k:=a(a+1)\cdots(a+k-1)$, {with $(a)_0=1$}.
\end{lemma}
\begin{proof}
It is easy to deduce that
\begin{equation}
\frac{\partial_{\theta_m} \cdots \partial_{\theta_1}
m_{\Lambda}(x;\theta)}{V_m(x)}  = s_{\gamma}(x_1,\dots,x_m)
m_{\La^s} (x_{m+1},\dots, x_{N})
\end{equation}
where $s_{\gamma}(x_1,\dots,x_m)$ is a Schur function.
The result is then immediate from the {evaluations} of
Schur \cite[Ex. 4. p. 45]{Mac} (or \cite[Theo. 5.4 and Prop. 1.2]{Stan}) and monomial functions \cite[Eq. 33]{Stan}.
\end{proof}

\begin{corollary} \label{coro22}
Let $\La_\mathrm{max}=\La_\mathrm{min}'$ be the
largest superpartitions  of degree $(n|m)$ in the dominance ordering
{(i.e., $\La_\mathrm{max}^a=\delta_m+(\ell_{n,m})$
and $\La_\mathrm{max}^s=0$). }
 Then
\begin{equation} \label{eqevalmin}
\Sp_{N,m}\,[m_{\La_\mathrm{min}}]=
\frac{(N-m-\ell_{n,m}+1)_{\ell_{n,m}}}{\ell_{n,m}!}\end{equation}
and \begin{equation}\Sp_{N,m}\,[m_{\La_\mathrm{max}}]=
{\frac{(m)_{\ell_{n,m}}}{\ell_{n,m}!}}
\end{equation}
\end{corollary}

The following technical result will be needed in the proof of the main theorem
of this section.
\begin{lemma} \label{LemmaJn}
 For $r,s$ nonnegative integers such that
$r \leq s$, we have that
\begin{equation}
\left[ {\partial^r_{x_1}  \partial_{\theta_1} J_{\tilde s}(x;\ta)} \right]_{x_1=\cdots=x_N=1} = \frac{s!}{(s-r)!} \alpha^{s} (1/\alpha+1)_r (N/\alpha+r+1)_{s-r}
\end{equation}
\end{lemma}
\begin{proof}  From \cite[Eq. 3.11]{DLMadv}  we have
\begin{equation} \label{eq515}
\sum_{s\geq 0} t^s \left[ g_s(x) + \tau \tilde g_s(x;\theta) \right] =
\prod_{i \geq 1} \frac{1}{(1-t x_i -\tau \theta_i)^{1/\alpha}}
\end{equation}
where the polynomials $g_s(x)$ and $\tilde g_s(x;\theta)$ are up to a constant
equal to $J_s(x)$ and $J_{\tilde s}(x;\ta)$ respectively:
\begin{equation}
J_{s}(x) = s! \, \alpha^s g_s(x) \qquad {\rm and} \qquad J_{\tilde s}(x;\theta)
= s! \, \alpha^{s+1} \tilde g_s(x;\theta)
\end{equation}
Taking the coefficient of $\tau$ on each side of \eqref{eq515} and using
\begin{equation}
\frac{\partial}{\partial {\theta_1}}(1-tx_1 -\tau \theta_1)^{-1/\alpha}
 =-\frac{\tau}{\alpha}
(1-tx_1-\tau \theta_1)^{-1/\alpha-1} \, ,
\end{equation}
we obtain
\begin{equation}
\sum_{s\geq 0} \frac{t^s}{s!\alpha^{s+1} }\, \partial_{\theta_1}
J_{\tilde s}(x;\ta)  =
\frac{1}{\alpha} (1-tx_1)^{-1/\alpha-1} (1-tx_2)^{-1/\alpha} \cdots
(1-tx_N)^{-1/\alpha} \, .
\end{equation}
This readily implies
\begin{align}
\left[\sum_{s\geq 0} \frac{t^s}{s! \alpha^{s+1}}
\partial^r_{x_1} \partial_{\theta_1} J_{\tilde s}(x;\ta) \right]_{x_1=\cdots=x_N=1}  & = \frac{t^r}{\alpha}
(1/\alpha+1)_{r} (1-t)^{-N/\alpha-r-1} \\
& =
\frac{1}{\alpha}(1/\alpha+1)_{r}
\sum_{m \geq 0} \frac{(N/\alpha+r+1)_{m}}{m!} t^{m+r}
\end{align}
The lemma then follows by taking the coefficient of $t^s$ on both sides of
the equation.
\end{proof}

In contradistinction to the evaluation of Jack polynomials
without fermions, the presence of fermions requires the introduction of
skew diagrams.  If $\Lambda$ is a superpartition of fermionic degree $m$,
let $\mathcal{S}$ be such that
\begin{equation}\mathcal{S } \La =
\La^{\circledast}/(m,m-1,\ldots,1).\end{equation}
Fig. \ref{FigS} illustrates the action of $\S$ on the diagram of $\La$.
A rationale for the addition of the $m$ circles and the
removal of the staircase $(m,m-1,\ldots,1)$
is that by dividing the polynomial by the Vandermonde determinant, we decrease the degree in $x$ by $m(m-1)/2$.

\begin{figure}[h]\caption{Operator $\S$ }\label{FigS}
\begin{equation*}
\mathcal{S} \,:\quad {\tableau[scY]{&&&&\bl\tcercle{}\\&&&\\&&&\bl\tcercle{}\\&&\\&\bl\tcercle{}}}  \longmapsto \quad {\tableau[scY]{\bl&\bl&\bl & &\\\bl&\bl&&\\\bl&&&\\&&\\&}}
\end{equation*}
\end{figure}

\begin{theorem}\label{TheoSpecialI}Let $\La$ be a superpartition of fermionic
degree $m$, and define
\begin{equation}\label{eqblambda}
b_\La^{(\alpha,N)}:=\prod_{(i,j)\in \mathcal{S}\La} b_\La^{(\alpha,N)}(i,j)  :=\prod_{(i,j)\in \S\La}\left(N-(i-1)+\alpha(j-1)\right).\end{equation}
Then, for all $N\geq \ell=\ell(\Lambda)$,
 \begin{equation}\label{spegen}
 \Sp_{N,m}[J_\La]\,=\,b_\La^{(\alpha,N)}.\end{equation}
\end{theorem}
\begin{proof}  We proceed by {induction on $m$ and on
the dominance ordering for fixed $m$.}
The base case $m=0$ being known \cite{Stan}, we can assume that $m>0$.
We have from Lemma~\ref{LemmaSpecialIm},
\eqref{eqevalmin} and the normalization \eqref{eqnormalJ} of
the Jack polynomials in superspace $J_{\La}$ that $\Sp_{N,m}[J_\La]$
is a monic polynomial in $N$ of degree $|\La|-m(m-1)/2$.

From Proposition~\ref{PropSkew}, we have that
\begin{equation}
J_{\La}(x;\theta) = \sum_{\Omega} k_{\Om} (\alpha)
 J_{\La/\Om}(x_1;\theta_1)
J_{\Omega}(x_-;\theta_-) ,
\end{equation}
where $k_{\Om}(\alpha)$ are coefficients that do not depend on $N$,  and
where $(x_-;\ta_-)$ stands for the variables
$(x_2,\cdots,x_N;\ta_2,\cdots \ta_N)$.
Since $m>0$, we have from \eqref{Spasder} that
\begin{equation} \label{eqspecI}
\Sp_{N,m} \,[J_\La] =  \sum_{\Om} (-1)^{m-1}  k_{\Om} (\alpha)
\left[\frac{\partial_{x_1}^{m-1}\partial_{\theta_1}
J_{\La/\Om}(x_1;\theta_1)}{(m-1)!}  \right]_{x_1=1}
\, \Sp_{N-1,m-1}\,[ J_{\Omega} ] .
\end{equation}
It is easy to see that the monomial $m_{\Ga}(x_1,\theta_1)=0$ if $\ell(\Ga)>1$.
From the triangularity \eqref{Ptriangular}, we thus also have that
$J_{\Ga}(x_1,\theta_1)=0$ if $\ell(\Ga)>1$.  Therefore $J_{\La/\Om}$
can only be equal (up to a constant)
to $J_r$ or $J_{\tilde r}$, where $r=|\La|-|\Om|$. Hence,
for $\partial_{\theta_1} J_{\La/\Om}(x_1;\theta_1)$ to be non-zero,
$J_{\La/\Om}$ must be equal (up to a constant) to $J_{\tilde r}$.
Proposition~\ref{PropgII} and \eqref{defskew}
then imply  that $\partial_{\theta_1} J_{\La/\Om}(x_1;\theta_1) \neq 0$ only if
$\La/\Om$ is a horizontal
$\tilde r$-strip.  Therefore $\mathcal{S}\Om$ contains all the cells
of $\S\La$ that do not lie at the bottom of a column along with cells
$(1,m),(2,m-1),\dots,(m,1)$ (since $\Omega$ has
one less fermionic row than $\La$).  By induction on the number of
fermions $m$ we
thus have that  $\Sp_{N-1,m-1} [J_{\Omega}]$ is a polynomial in $N$ divisible by
\begin{align} \nonumber
\prod_{i=1}^m  \left(N-1-(i-1)+\alpha(m-i)\right)
& \prod_{
\begin{subarray}{c}
(i,j) \in \S\La \\
i \neq ({\La^{\circledast}}')_j
\end{subarray}}  \left(N-1-(i-1)+\alpha(j-1)\right) \\ \label{eqRHS}
& = \prod_{
\begin{subarray}{c}
(i,j) \in \S\La\\
(i,j)\not \in \{ (1,m+1),\dots,(1,\La^{\circledast}_1) \}
\end{subarray}}  \left(N-(i-1)+\alpha(j-1)\right)
\end{align}
and therefore $\Sp_{N,m} [J_\La]$ is also divisible by \eqref{eqRHS}.
We thus have left to show that $\Sp_{N,m} [J_\La]$ is divisible
by
\begin{equation} \label{Termadiv}
\prod_{(i,j) \in \{ (1,m+1),\dots,(1,\La^{\circledast}_1) \} }
 \left(N-(i-1)+\alpha(j-1)\right)=
 (N+m\alpha) \cdots (N+(\La^{\circledast}_1-1)\alpha)
\end{equation}

Suppose that the
first row of $\La^{\circledast}$ is fermionic ($\La^*_1= \La^{\circledast}_1-1=s$)
and let $\Om$ be such that $\Om^*=(\La^*_2,\La^*_3,\dots)$ and
$\Om^{\circledast}=(\La^{\circledast}_2,\La^{\circledast}_3,\dots)$.  Then
\begin{equation} \label{eqdecomp1}
J_{\tilde s} J_{\Om} = g_{\Om \tilde s}^{\La} \frac{v_{\Om} v_{\tilde s}}{\|P_\La\|^2 v_{\La}} J_{\La} +
\sum_{\Gamma > { \La}}   g_{\Om \tilde s}^{\Ga} \frac{v_{\Om} v_{\tilde s}}
{\|P_\Ga\|^2 v_{\Ga}} J_{\Ga}
\end{equation}
where $g_{\Om \tilde s}^{\La} \neq 0$ by Proposition~\ref{Propg}
{ since $\Om^* \cup s = \La^*$ and  $\Om^{\circledast} \cup
(s+1) = \La^{\circledast}$.}

We first prove that
 $\Sp_{N,m} [J_{\tilde s} J_{\Om}]$ is divisible
by \eqref{Termadiv}.  The proof is not entirely straightforward and
will rely on Lemma~\ref{LemmaJn}.

Using \eqref{Spasder}, we can write
\begin{align} \label{eqevaluS}
 \Sp_{N,m}\, [J_{\tilde s} J_{\Om} ]&=
\left[\partial^{\delta}_x \partial_{\theta_m} \cdots \partial_{\theta_1}
(J_{\tilde s} J_{\Om}) \right]_{x_1=\cdots=x_N=1}\nonumber\\
&= \left[ \partial_x^{\delta}
\sum_{i=1}^N (-1)^{m-i}  (\partial_{\theta_i} J_{\tilde s})
 \partial_{\theta_m} \cdots  \widehat{\partial_{\theta_i}} \cdots
\partial_{\theta_1} J_{\Om}  \right]_{x_1=\cdots=x_N=1},
\end{align}
where $ \widehat{\partial_{\theta_i}}$ means that $\partial_{\theta_i}$ is omitted. Consider the term $i=1$.
The expression
$\partial_{\theta_m} \cdots \partial_{\theta_2} J_{\Om}$
is antisymmetric in $x_2,\dots,x_m$ and thus divisible by
$\prod_{2\leq i <j \leq m}(x_i-x_j)$. This implies that
\begin{equation}
\partial_{\theta_m} \cdots \partial_{\theta_2} J_{\Om}=H(x_1,\dots,x_N) \prod_{2\leq i <j \leq m}(x_i-x_j)
\end{equation}
for a certain polynomial $H(x_1,\dots,x_N)$, from which we get that
\begin{align}
& \left[ \partial_x^{\delta}\bigl(
\partial_{\theta_1}(J_{\tilde s})
\partial_{\theta_m} \cdots \partial_{\theta_2} J_{\Om} \bigr)
\right]_{x_1=\cdots=x_N=1} \nonumber \\
& = \left[ \partial_x^{\delta} \Bigl(
\partial_{\theta_1}(J_{\tilde s})
H(x_1,\dots,x_N) \prod_{2\leq i <j \leq m}(x_i-x_j) \Bigr)
\right]_{x_1=\cdots=x_N=1} \nonumber \\ \label{lastpiece}
& = \left[ \frac{1}{(m-1)!}\partial^{m-1}_{x_1} \Bigl(
\partial_{\theta_1}(J_{\tilde s})
H(x_1,\dots,x_N) \Bigr)
\, \partial_{x_-}^{\delta_{m-1}} \prod_{2\leq i <j \leq m}(x_i-x_j)
\right]_{x_1=\cdots=x_N=1}
\end{align}
where
\begin{equation}
\partial_{x_-}^{\delta_{m-1}} :=  \left[ \prod_{i=1}^{m-2} \frac{1}{(m-i-1)!} \right]
\partial_{x_2}^{m-2} \partial_{x_3}^{m-3} \cdots \partial_{x_{m-1}}.
\end{equation}
By Lemma~\ref{LemmaJn}, we have that $\partial_{x_1}^r \partial_{\theta_1}
J_{\tilde s}$ is divisible by \eqref{Termadiv} for any
$r \leq m-1$.  Hence from \eqref{lastpiece}
\begin{equation}
\left[ \partial_x^{\delta}\bigl(
\partial_{\theta_1}(J_{(s;)})\partial_{\theta_m} \cdots \partial_{\theta_2} J_{\Om}
\bigr)
\right]_{x_1=\cdots=x_N=1}
\end{equation}
is also  divisible by \eqref{Termadiv}.
By symmetry of $J_{\tilde s} J_{\Om}$ under the exchange
$x_1 \leftrightarrow x_i$, $\theta_1 \leftrightarrow \theta_i$,
it follows that
every term in the right-hand side of \eqref{eqevaluS}
is divisible by \eqref{Termadiv}, and
consequently so is $\Sp_{N,m} [J_{\tilde s} J_{\Om}]$.

Let $\La=\La_{\rm max}=\La_{\rm min}'$
be the largest
superpartition of $(n|m)$ in dominance order.
We now
show that
$\Sp_{N,m}[J_{\La_{\rm max}}]$
is divisible by  \eqref{Termadiv}, the starting point in our induction
process for fixed $m > 0$.
Observe that the first row of the diagram of
$\La_{\rm max}$ is fermionic since the first column of
the diagram of $\La_{\rm min}$ is fermionic.
If $m=1$,
then ${\La_{\rm max}}=\tilde n$ and
\begin{equation}
\Sp_{N,1}[J_{\tilde n}]= \bigl[ {\partial_{\theta_1}}
J_{\tilde n} \bigr]_{x_1=\cdots=x_N=1} = (N+\alpha) \cdots (N+n\alpha)
\end{equation}
by Lemma~\ref{LemmaJn}.  Therefore $\Sp_{N,1}[J_{\La_{\rm max}}]$
is divisible by  \eqref{Termadiv} in that case.
If $m>1$,  using \eqref{eqdecomp1} with $\La={\La_{\rm max}}$
gives that $J_{\La_{\rm max}}$ is equal to
$J_{\tilde s} J_\Om$ up to a
non-zero constant.  Since we have established that
$\Sp_{N,m}[J_{\tilde s} J_\Om]$
is divisible by \eqref{Termadiv}, our claim follows.

Now we consider again the general situation in \eqref{eqdecomp1}.
The polynomials $J_\Ga$ that appear in the sum in the right-hand-side
are such that $\Ga>\La$.  The first row
of $\Ga^{\circledast}$ is then  not smaller than the first
row of $\La^{\circledast}$.
By induction on the dominance ordering
we thus have that $\Sp_{N,m} [J_{\Ga}]$
is divisible by \eqref{Termadiv}.   Since we have established that
$\Sp_{N,m}[J_{\tilde s} J_\Om]$
is divisible by \eqref{Termadiv}, so is $\Sp_{N,m} [J_{\La}]$ given
that its coefficient in \eqref{eqdecomp1} is non-zero.

Finally, suppose that the
first row of $\La^{\circledast}$ is bosonic ($\La^*_1= \La^{\circledast}_1=s$)
and let $\Om$ be such that $\Om^*=(\La^*_2,\La^*_3,\dots)$ and
$\Om^{\circledast}=(\La^{\circledast}_2,\La^{\circledast}_3,\dots)$.  Then
\begin{equation} \label{eqdecomp2}
J_{s} J_{\Om} = g_{\Om s}^{\La} \frac{v_{\Om} v_{s}}{\|P_\La\|^2 v_{\La}} J_{\La} +
\sum_{\Gamma > \Om}   g_{\Om s}^{\Ga} \frac{v_{\Om} v_{s}}{\|P_\Ga\|^2 v_{\Ga}} J_{\Ga}
\end{equation}
where $g_{\Om s}^{\La} \neq 0$ by Proposition~\ref{Propg}
{ since $\Om^* \cup (s) = \La^*$ and  $\Om^{\circledast} \cup (s) =
\La^{\circledast}$.} Again we have by induction
 on the dominance ordering that $\Sp_{N,m} [J_{\Ga}]$
is divisible by \eqref{Termadiv} for every $J_{\Ga}$ appearing in the sum.
Since $\Sp_{N,m} [J_{s} J_{\Om}]= \Sp_{N,m} [J_{s}]
\, \Sp_{N,0} [J_{\Omega}]$ and   $\Sp_{N,0} [J_{s}]$ is divisible by
\eqref{Termadiv} (see \eqref{eqblambda} in the case $\Lambda=s$),
we have that
$\Sp_{N,m} [J_{\La}]$
is also divisible by \eqref{Termadiv}.

The polynomials   \eqref{eqRHS}
and \eqref{Termadiv} are relatively prime (since their zeroes are distinct for generic values of $\alpha$)
and we have shown that $\Sp_{N,m} [J_{\La}]$ is divisible by both.  Therefore
$
\Sp_{N,m} [J_{\La}]$ is divisible by their product which is equal to \eqref{eqblambda}, a monic polynomial
of degree $|\La|-m(m-1)/2$.  Since we have seen that
$\Sp_{N,m} [J_{\La}]$
is monic and of degree $|\La|-m(m-1)/2$, the theorem follows.
\end{proof}

Let us end this subsection with an example illustrating
formula (\ref{spegen}).  Consider for instance $\Lambda=(3,1;2)$, for which
\begin{equation}
\mathcal{S}(3,1;2)={\tableau[scY]{\bl&\bl&&\\\bl&\\&}}
\end{equation}
By filling the cells $(i,j)$ of $\S\La$ with the factors $N-(i-1)+\alpha(j-1)$, one gets
\begin{equation} {\small{\tableau[mcY]{
\bl {\mbox{\tiny}} &\bl{\mbox{\tiny }} &{\mbox{\tiny $\bar{0}\!+\!2\alpha$}}&{\mbox{\tiny $\bar{0}\!+\!3\alpha$}} \\
\bl &{\mbox{\tiny $\bar{1}\!+ \!\alpha$}}
\\ {\mbox{\tiny$\bar{2}$}}   & {\mbox{\tiny$\bar{2}\!+\!\alpha$}}
}} } \end{equation}
where $\bar{k}=N-k$ with the understanding that $N\geq 3$.  Thus
\begin{equation}\Sp_{N,2}\,[J_{(3,1;2)}]=b_{(3,1;2)}^{(\alpha,N)}={ (N+2\alpha)(N+3\alpha)(N-1+\alpha)(N-2)(N-2+\alpha)}.
\end{equation}

\subsection{Second evaluation formula}\label{spe2}
Remarkably, if the Jack superpolynomial $J_\La(x,\theta)$ has
non-zero fermionic degree, the evaluation $E_{N-1,m-1}$ of
$[\partial_{\theta_N} J_{\La}]_{x_N=0}$ is very similar to that of $J_{\La}$
even though its expansion
in terms of Jack superpolynomials can involve many terms.  We will
refer to the evaluation of $[\partial_{\theta_N} J_{\La}]_{x_N=0}$ has our
second evaluation formula.  To simplify the notation, we will define
\begin{equation}\label{StS}\Spt_{N,m}[F(x,\theta)]
:=\Sp_{N-1,m-1}\left[(-1)^{m-1}\partial_{\theta_N} \,F(x;\theta)\right]_{x_N=0} \, .
\end{equation}
Before getting to the derivation of the explicit form of the second
evaluation formula on Jack polynomials in superspace,
we must introduce another operation on partition.  Let $\La$ be a superpartition
of fermionic degree $m$.  Then $\tilde{\mathcal{S}}$ on $\La$ is defined as
\begin{equation}\tilde{\mathcal{S}} \La =\La^*/(m-1,m-2,\ldots,0).\end{equation}
See Fig.\ \ref{FigStilde} for a diagrammatic illustration of this definition.
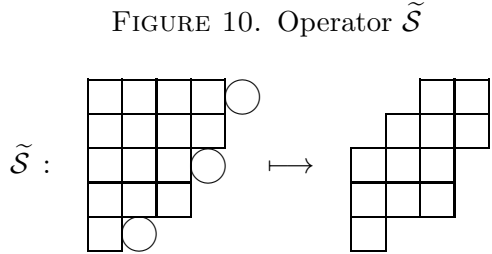
\begin{figure}[h]\caption{Operator $\widetilde{\mathcal{S}}$ }\label{FigStilde}
\begin{equation*}
\widetilde{\mathcal{S}} \,:\quad {\tableau[scY]{&&&&\bl\tcercle{}\\&&&\\&&&\bl\tcercle{}\\&&\\&\bl\tcercle{}}}  \longmapsto \quad {\tableau[scY]{\bl&\bl & & &\bl\\\bl& &&\\ &&&\bl\\&&\\&\bl}}
\end{equation*}
\end{figure}

\begin{lemma}\label{lemma19}
Let $\La=(\La_1,\ldots,\La_m;\La_{m+1},\ldots,\La_\ell)$ be a superpartition.
Then
\begin{equation}\Spt_{N,m}\,[m_\La]=
\left\{
\begin{array}{ll}
\Sp_{N-1,m-1} [m_{\La_-}] & {\rm if~} \Lambda_m=0 \\
0 & {\rm otherwise}
\end{array}
\right.
\end{equation}
where $\La_-=(\La_1,\ldots,\La_{m-1};\La_{m+1},\ldots)$.
\end{lemma}
\begin{proof} The lemma follows immediately from
\begin{equation}\left[\partial_{\theta_N} \,m_\La(x;\theta)\right]_{x_N=0}=
\left\{
\begin{array}{ll}
(-1)^{m-1}m_{\La_-}(x_-;\theta_-) & {\rm if~} \Lambda_m=0 \\
0 & {\rm otherwise}
\end{array}
\right.
\end{equation}
where $(x_-;\theta_-)=(x_1,\ldots,x_{N-1};\theta_1,\ldots,\theta_{N-1})$.
\end{proof}

\begin{theorem}\label{TheoSpecialII}Let $\La$ be a superpartition
of fermionic degree $m>0$.
 Let also
\begin{equation}\tilde{b}_\La^{(\alpha,N)}:= \prod_{(i,j)\in \tilde{\mathcal{S}}\La}\tilde{b}_\La^{(\alpha,N)} (i,j):= \prod_{(i,j)\in \tilde{\mathcal{S}}\La}\left(N-1-(i-1)+\alpha(j-1)\right).\end{equation}
Then, for all $N\geq \ell=\ell(\La)$,
 \begin{equation}\Spt_{N,m}\,[J_\La]=\tilde{b}_N^{(\alpha,N)}.\end{equation}
\end{theorem}
\begin{proof}
As in the proof of Theorem~\ref{TheoSpecialI}, we first prove that
$\Spt_{N,m} [J_{\La}]$ is a monic polynomial in $N$ of degree
$\ell_{n,m}=|\La|-m(m-1)/2$.
From Lemmas~ \ref{LemmaSpecialIm} and \ref{lemma19}, we have
that the degree in $N$ of $\Spt_{N,m} [J_{\La}]$ is given by
the degree in $N$ of $\Spt_{N,m}[m_{\La_{\rm min}}]$.  Using
\begin{equation}
\Spt_{N,m}[m_{\La_{\rm min}}] = \Sp_{N-1,m-1} [m_{(\La_{\rm min})-}] =
\frac{(N-\ell_{n,m}+m)_{\ell_{n,m}}}{\ell_{n,m}!}
\end{equation}
the claim then follows from
the normalization \eqref{eqnormalJ} of $J_{\La}$ .
From Proposition~\ref{PropSkew}, we have that
\begin{equation} \label{eqimportante}
\Spt_{N,m} \,[J_\La(x,\theta)]=
\sum_{\Omega} k_{\Om} (\alpha)
\,\Sp_{N-1,m-1} \, [ J_{\Om}(x_-,\theta_-)]
\left[\partial_{\theta_N}
J_{\La/\Om}(x_N,\theta_N)\right]_{x_N=0}
\end{equation}
where $k_{\Om}(\alpha)$ are coefficients that do not depend on $N$.
Since
\begin{equation}
\left[\partial_{\theta_N} J_{\Ga}(x_N,\theta_N)\right]_{x_N=0} \neq 0
\quad \iff \quad \Ga=\tilde 0
\end{equation}
we have that $\left[\partial_{\theta_N}
J_{\La/\Om}(x_N,\theta_N)\right]_{x_N=0} \neq 0$ only if
$g^{\La}_{\Om \tilde 0} \neq 0$, that is only if
$\La/\Om$ is a horizontal $\tilde 0$-strip by
Proposition~\ref{PropgII}.
If $\La/\Om$ is a horizontal $\tilde 0$-strip  then
the diagram of $\Om$ corresponds to the diagram of $\La$ with one of
its circles removed.  It is then immediate that
$\Sp_{N-1,m-1} [J_{\Om}(x_-, \theta_-)]$ is divisible by
$\tilde{b}_\La^{(\alpha,N)}$
given that  $\tilde \S (\Lambda)$ is contained
in $\S \Omega=\Omega^{\circledast}/(m-1,\dots,1)$.
Since every term in the right-hand side of
\eqref{eqimportante} is divisible by $\tilde{b}_\La^{(\alpha,N)}$, so is
$\Spt_{N,m} [J_\La]$.  The theorem then
follows given that
$\tilde{b}_\La^{(\alpha,N)}$ is a monic polynomial in $N$ of degree
$|\La|-m(m-1)/2$.
\end{proof}

Let us illustrate the second evaluation formula with the
superpartition $\La=(3,0;2,1)$.  The evaluation amounts to filling
the cells of $\tilde{\mathcal{S}} (3,0;2,1)$ with
the numbers  $\tilde{b}_{(3,0;2,1)}^{(\alpha,N)} (i,j)$.
Using $\bar k= N-k$, this corresponds to the filling
\begin{equation} {\small{\tableau[mcY]{\bl
{\mbox{\tiny$$}}&{\mbox{\tiny $\bar{1}\!+\!\alpha$}} &{\mbox{\tiny$\bar{1}\!+\!2\alpha$}}\\
{\mbox{\tiny$\bar{2}$}}&{\mbox{\tiny$\bar{2}\!+\!\alpha$}} \\
{\mbox{\tiny $ \bar{3}$}}
}} }
\end{equation}
which gives
\begin{equation}\Spt_{N,2}\,[J_{(3,0;2,1)}]=(N-1+\alpha)(N-1+2\alpha)(N-2)(N-2+\alpha)(N-3).\end{equation}

\subsection{Homomorphisms}
Evaluations
over rings are usually  defined in algebra as ring homomorphisms.  As an example, consider the ring $\mathrm{Sym}$ formed by the symmetric functions in the indeterminates $x=(x_1,x_2,\ldots)$
with coefficients in the field $\mathbb{Q}(\alpha)$.  Then, Macdonald \cite[Section I.2]{Mac} defines the standard evaluation of symmetric functions as the homomorphism
$\varepsilon_X :\,\mathrm{Sym}\to \mathbb{Q}(\alpha)[X]
$
such that
\begin{equation}\label{defevalhomo2}\varepsilon_X(p_n)=X,\quad  \forall\, n\geq 1.
\end{equation}
In the case of the usual Jack symmetric functions $J_\la$, it is possible to show that \cite[Eq. 10.25]{Mac}
\begin{equation}\varepsilon_X(J_\lambda)=\prod_{(i,j)\in\la}(X-(i-1)+\alpha(j-1)).
\end{equation}
Setting $X=N$ in the last equation, we recover Stanley's result for the evaluation of a Jack symmetric function  at $x_1=\ldots=x_N=1$ and $x_{N+1}=x_{N+2}=\ldots=0$ \cite[Theo.\ 5.4]{Stan}.

Obviously, the evaluation operator $E_{N,m}$ introduced in  \eqref{eqdefENm} is not a homomorphism.
Rather, it is a linear map from the vector space $\mathscr{R}^{n,m}_N$ to the ring  $\mathbb Q(\alpha)[N]$.
It is nevertheless possible to define an evaluation homomorphism $\mathcal{E}_{X,Y}^M$ on the whole ring   $\mathscr{R}$ of symmetric superfunctions
with coefficient in $\mathbb{Q}(\alpha)$. As we show below,  $\mathcal{E}_{X,Y}^M$ generalizes the homomorphism $\varepsilon_X$ defined above and
connects with  $E_{N,m}$ in the special case where $X=N$ and $Y=M=m$.

Now let $R$ and $R[\phi_1,\ldots,\phi_M]$ respectively denote the polynomial ring $\mathbb{Q}(\alpha)[X,Y]$ and the ring of polynomials in the Grassmann variables
$\phi_1,\ldots,\phi_M$ with coefficients in $R$. We define \begin{equation} \mathcal{E}_{X,Y}^M:\,\mathscr{R}\to R[\phi_1,\ldots,\phi_M]\end{equation} as the homomorphism such that \begin{equation} \mathcal{E}_{X,Y}^M(p_{n+1})=X, \qquad \mathcal{E}_{X,Y}^M(\tilde{p}_{n})=\sum_{i=1}^M\phi_i \varepsilon_Y(h_{n+i-M}),\qquad \forall \,n\geq 0,\end{equation} where $h_r$ denotes the complete symmetric function of degree $r$ \cite[p.\ 21]{Mac}.  By convention, $h_r$ is equal to 1 for $r=0$ and to 0 whenever $r<0$.  On readily shows \cite{Mac}
that if $r>0$, \begin{equation} \varepsilon_Y(h_r)=\binom{Y+r-1}{r}=(-1)^r\binom{-Y}{r}. \end{equation}

\begin{proposition}\label{propevalhomo} For any symmetric
function in superspace
$F(x,\theta)$ of fermionic degree $m$, we have
  \begin{equation} \mathcal{E}_{N,m}^m(F(x,\theta))
=\phi_1\cdots \phi_m E_{N,m}(F(x,\theta)).
\end{equation} \end{proposition}
\begin{proof}
It is enough to show that the proposition holds when
$F(x,\theta)=p_{\Lambda}$, where $\La=(\La^a;\La^s)$.
Let us first recall a basic lemma: if  $F_i=\sum_{j} \phi_j A_{i,j}$, where $\phi_j$ and $A_{i,j}$  respectively denote anticommutative and commutative variables, then
  \begin{equation}\label{eqformuladet} F_{1}\cdots F_{m}= \sum_{1\leq j_1<\ldots<j_m\leq M}\phi_{j_1}\cdots\phi_{j_m}\det\Big[A_{i,j}\Big]_{\substack{i= 1,\ldots,m\\j=j_1,\ldots,j_m}}.\end{equation}
Applying formula \eqref{eqformuladet} to the case
\begin{equation}
F_{1}\cdots F_{m} = \tilde p_{\Lambda_1^a} \cdots \tilde p_{\Lambda_m^a}
\end{equation}
which corresponds to  $\phi_i=\theta_i$ and $A_{i,j}=x_j^{\Lambda^a_i}$, we get
\begin{align}p_\La(x;\theta)&=\sum_{1\leq j_1<\ldots<j_m\leq m}\theta_{j_1}\cdots\theta_{j_m}\,a_{\La^a}(x_{j_1},\ldots,x_{j_m})\,p_{\La^s}(x_1,\ldots,x_N)\\
&=\sum_{1\leq j_1<\ldots<j_m\leq m}\theta_{j_1}\cdots\theta_{j_m}\,V_m(x_{j_1},\ldots,x_{j_m})s_{\La^a-\delta_m}(x_{j_1},\ldots,x_{j_m})
\,p_{\La^s}(x_1,\ldots,x_N),\end{align}
where $a_{\lambda}(x_1,\dots,x_m)=\sum_{\sigma \in S_m} {\rm sgn}(\sigma) x^{\lambda_1}_{\sigma(1)}
\cdots x^{\lambda_m}_{\sigma(m)}$ and  where the second equality follows from  the
standard definition of Schur functions.
Consequently
  \begin{equation}E_{N,m}(p_\La)= N^{\ell(\La^s)}s_{\La^a-\delta_m}(1^m).\end{equation}
It thus remains to prove that
\begin{equation}\mathcal{E}^m_{N,m}(p_\La)=\phi_1\ldots\phi_m N^{\ell(\La^s)}s_{\La^a-\delta_m}(1^m).\end{equation}
Returning to \eqref{eqformuladet} and setting $A_{i,j}=
\varepsilon_Y(h_{\Lambda_i^a+j-M})$, we get
 \begin{equation} \mathcal{E}_{X,Y}^M(p_{\La})=X^{\ell(\La^s)}
\sum_{1\leq j_1<\ldots<j_m\leq M}\phi_{j_1}\cdots\phi_{j_m}\det\Big[
\varepsilon_Y(h_{\Lambda_i^a+j-M})
\Big]_{\substack{i= 1,\ldots,m\\j=j_1,\ldots,j_m}}.
\end{equation}
  When $M=m$, the sum on the RHS produces only one term, which is equal to \begin{equation}\phi_1\cdots\phi_m \det\Big[ \varepsilon_Y(h_{\La^a_i+j-m})\Big]_{\substack{i= 1,\ldots,m\\j=1,\ldots,m}}.\end{equation}
Now, according to the Jacobi-Trudi formula \cite[Eq.  3.4]{Mac},
\begin{equation} s_{\la-\delta_m}= \det\Big[ h_{\la_i+j-m}
\Big]_{\substack{i= 1,\ldots,m\\j=1,\ldots,m}},
\end{equation}
 where it is understood that we are working with symmetric polynomials in $m$ variables.  Hence \begin{equation} \mathcal{E}_{X,Y}^m(p_{\La})=
X^{\ell(\La^s)}\phi_1\cdots\phi_m \varepsilon_Y (s_{\La^a-\delta_m}),
\end{equation}
which reduces to the desired equation for $X=N$ and $Y=m$.
\end{proof}

\section{Normalization of the Jack polynomials in superspace}
\label{Snorm}

The simplest way of writing
the normalization of the standard Jack polynomials is in terms
of the upper and lower hook-lengths \cite{Stan}. The same is true for the norm of the Jack superpolynomials: it is expressed in terms of the superpartition hook-lengths.  Recall that leg-lengths and arm-lengths were
defined in Section~\ref{sectintro}.
\begin{definition}\label{defHook}
The upper and lower hook-lengths of $s\in\La$
 are respectively given by
\begin{equation}
h^{(\alpha)}_\La(s)={l}_{\Lambda^{\circledast}}(s)+\alpha(a_{\La^*}(s)+1)\qquad \text{and}\qquad
h^\La_{(\alpha)}(s)=l_{\La^*}(s)+1+\alpha\,{a}_{\Lambda^{\circledast}}(s).
\end{equation}
\end{definition}

\begin{lemma}\label{lemmahook}For any diagram $\La$, the two hook-lengths are related by
\begin{equation}h_{(\alpha)}^\La(i,j)=\alpha \,h_{\La'}^{(1/\alpha)}(j,i).\end{equation}
\end{lemma}
\begin{proof}This is an immediate consequence of the identity $l_\la(i,j)=a_{\la'}(j,i)$.\end{proof}

Now recall Definition \ref{defJackJ}
 relating the non-monic
 $J_\La$ to its monic counterpart $P_\La$. If we want to re-express Proposition \ref{PropFactoI} in terms of $J_\La$, this will necessarily involve a non-trivial proportionality factor since the Jack superpolynomials
on the two sides of the equation are different. Stated precisely,
 there must exist a rational function $r_\La(\alpha)$
 such that
\begin{equation} \label{eqdefr} J_\La(x_1,\ldots,x_\ell;\theta_1,\ldots,\theta_\ell)=r_\La(\alpha) \, x_1\cdots x_\ell \,J_{\mathcal{C}\La}(x_1,\ldots,x_\ell;\theta_1,\ldots,\theta_\ell),
\end{equation}
if $\La$ is a superpartition of length $\ell$  whose first column is bosonic.
Similarly,
{it follows } from Proposition~\ref{PropFactoII}
that there exists  a
$\tilde{r}_\La(\alpha)$ such that \begin{equation}\label{eqdefrtilde}
(-1)^{m-1}
\Bigl[ \partial_{\theta_\ell}\,J_\La(x_1,\ldots,x_\ell;\theta_1,\ldots,\theta_\ell)
\Bigr]_{x_\ell=0}
=
\tilde{r}_\La(\alpha)J_{\widetilde{\C}\La}(x_1,\ldots,x_{\ell-1};\theta_1,\ldots,\theta_{\ell-1}).
\end{equation}
when
$\ell(\La^\circledast)=\ell(\La^*)+1$. The rationale for introducing $r_\La(\aa)$ and $\tilde r_\La(\aa)$ is that they are the building blocks of the norm expression. Actually, the hook-lengths do appear in the norm via these proportionality factors.

 The next two propositions give the relation between
 $r_\La(\alpha)$ and $\tilde{r}_\La(\alpha)$ with the lower and upper hook-lengths respectively. They imply in particular that  $r_\La(\alpha)$ and $\tilde{r}_\La(\alpha)^{-1}$  are
 polynomials in $\alpha$.

\begin{proposition} \label{proprla}
Let $\La$ be a superpartition such that $\La_i>0$ for all $1\leq i\leq\ell$ and let $r_\La(\alpha)$  {be defined by \eqref{eqdefr}}. Then
\begin{equation}r_\La(\alpha)=\prod_{i=1}^\ell h^\La_{(\alpha)}(i,1).\end{equation}
\end{proposition}
\begin{proof}
We apply the evaluation $\Sp_{\ell,m}$ on both sides of  \eqref{eqdefr} with $m$ standing for the fermionic degree of $\La$.  This yields
\begin{equation}r_\La(\alpha)=\frac{\Sp_{\ell,m}\,[J_\La]}{\Sp_{\ell,m}\,[J_{\C\La}] }
.\end{equation}
Here it is crucial that the first column of $\Lambda$ be bosonic
to insure that the action of $\C$ on $\La$ is well defined.
Now,
Theorem \ref{TheoSpecialI} (i.e., \eqref{spegen}) implies that
\begin{equation}
r_\La(\alpha)=\frac{b_\La^{(\alpha,\ell)}}{b_{\C\La}^{(\alpha,\ell)}}.
\end{equation}
From the definition of $b_\La^{(\alpha,N)}$ given in \eqref{eqblambda}, it then follows that
\begin{equation}\label{calp1}r_\La(\alpha)=\prod_{(i,j)\in \mathcal{S}\La/\mathcal{S}(\C\La)}\left(\ell-(i-1)+\alpha(j-1)\right) = \prod_{\substack{(i,j)\in \mathcal{S}\La\\j= \La^\cd_i} }\left(\ell-(i-1)+\alpha(j-1)\right),\end{equation}
where $\La^\cd_i$ denotes the number of cells in the $i$th row of the diagram $\La^\cd$.
 However, $\ell-(i-1)=l_{\La^*}(i,1)+1$ while $j-1={a}_{\La^{\circledast}}(i,1)$ if $j=\La^\cd_i$.  Thus,
\begin{equation}\label{calp2}r_\La(\alpha)=\prod_{i=1}^\ell\left(l_{\La^*}(i,1)+1+\alpha\,{a}_{\La^\circledast}(i,1)\right),\end{equation}
and the proof is complete. \end{proof}

\begin{proposition} \label{proprlat}
Let $\La$ be a superpartition such that $\ell(\La^\circledast)=\ell(\La^*)+1$. Moreover, let ${\mathrm{fr}}(\La)$ denote the set of fermionic rows in the diagrams of $\La$ and let $\tilde{r}_\La(\alpha)$
be the function introduced in \eqref{eqdefrtilde}.   Then
\begin{equation}\tilde{r}_\La(\alpha)=\prod_{(i,1)\in{\mathrm{fr}}(\La)} \frac{1}{h^\La_{(\alpha)}(i,1)}.\end{equation}
\end{proposition}
\begin{proof}We  set $N=\ell=\ell(\La^\circledast)$ and $m=\overline{\underline{\La}}$.  Applying the evaluation $E_{\ell-1,m-1}$
on both sides of \eqref{eqdefrtilde} and using relation (\ref{StS}),
we get
\begin{equation}\tilde{r}_\La(\alpha)=\frac{\Spt_{\ell,m}\,[J_\La]}
{\Sp_{\ell-1,m-1}\,[J_{ \tilde{\C}\La}]}
=\frac{\tilde{b}_\La^{(\alpha,\ell)}}{b_{ \tilde{\C}\La}^{(\alpha,\ell-1)}} \, .
\end{equation}
The second equality follows from Theorem \ref{TheoSpecialI}.
Comparing the diagrams of $\tilde{\mathcal{S}}\La$ and $\mathcal{S} (\tilde{\C}\La)$, we find
\begin{multline}\label{ctalpa}\tilde{r}_\La(\alpha)= \! \!
\prod_{(i,j)\in\mathcal{ S}(\tilde{\C}\La)/\tilde{\S}\La} \! \!
\frac1{\left(\ell-1-(i-1)+\alpha(j-1)\right)}
= \!\!
\prod_{\substack{(i,1)\in \mathrm{fr}(\La)\\ i \neq \ell}} \! \!
\frac1{\left(\ell-i+\alpha(\La^\cd_i-1)\right)} .
 \end{multline}
Again, we have $\ell-i=l_{\La^*}(i,1)+1$ while $\La^\cd_i-1={a}_{\La^\circledast}(i,1)$.  Moreover, since $l_{\La^*}(\ell,1)=0$ and $a_{\La^\circledast}(\ell,1)=0$, we can add the contribution of the square $(1,\ell)$ without modifying the result.
 Hence we can write
\begin{equation}\tilde{r}_\La(\alpha)=\prod_{(i,1)\in \mathrm{fr}(\La)}\left(l_{\La^*}(i,1)+1+\alpha\,{a}_{\La^\circledast}(i,1)\right)^{-1},\end{equation}
which is the desired result. \end{proof}

\begin{theorem}
Let $\mathcal{B}\La$ denote the bosonic content of $\La$, i.e., the set of squares
in $\La$ that do not appear at the same time in a row containing a circle and in a column containing a circle.  Then the coefficient  of $m_\La$ in $J_\La$ is
\begin{equation}v_\La(\alpha)=\prod_{(i,j)\in\mathcal{B}\La}h^\La_{(\alpha)}(i,j).\end{equation}
\end{theorem}
\begin{proof}  We proceed by induction on the size of $\Lambda^{\circledast}$.
If $\Lambda$ is the empty partition,
then $J_{\emptyset}=1$ and the result holds since $v_{\emptyset}=1$.
If the first column of $\Lambda$ is bosonic, then
from \eqref{eqdefr}, Proposition~\ref{PropFactoI}
 and Definition \ref{defJackJ}, we have
\begin{equation}v_\La(\alpha)=r_\La(\alpha) \,  v_{\C\La}(\alpha) \, .
\end{equation}
Using Proposition~\ref{proprla} and setting $\ell=\ell(\lambda)$,
we thus have by induction that
\begin{equation}v_\La(\alpha)
=\prod_{i=1}^\ell h^\La_{(\alpha)}(i,1)
\prod_{(i,j)\in\mathcal{B}(\C\La)}h^{\C\La}_{(\alpha)}(i,j) =
\prod_{(i,j)\in\mathcal{B}\La}h^\La_{(\alpha)}(i,j) \, ,
\end{equation}
and the result holds in that case.
If the first column of $\La$ is fermionic, we have from
 \eqref{eqdefrtilde},
Proposition~\ref{PropFactoII} and Definition \ref{defJackJ} that
\begin{equation}v_\La(\alpha)=\tilde{r}_\La (\alpha) \,
 v_{ \tilde{\C}\La}(\alpha) \, .
\end{equation}
Using Proposition~\ref{proprlat}, we then have by induction that
\begin{equation}v_\La(\alpha)=\prod_{(i,1)\in{\mathrm{fr}}(\La)}
\frac{1}{h^\La_{(\alpha)}(i,1)}
\prod_{(i,j)\in\mathcal{B}(\tilde{\C}\La)}h^{\tilde{\C}\La}_{(\alpha)}(i,j) =
\prod_{(i,j)\in\mathcal{B}\La}h^\La_{(\alpha)}(i,j) \, ,
\end{equation}
which proves the theorem.
\end{proof}

To illustrate the last {formula}, we consider the superpartition
\begin{equation}
  \La=(4,2,0;2)=\tableau[scY]{&&&&\bl\tcercle{}\\& &\bl\tcercle{}\\ &\\ \bl\tcercle{} }
\end{equation}
The bosonic content of $\La$ and its associated upper hook-lengths are given by
\begin{equation} {\small{\tableau[mcY]{\bl &{\mbox{\tiny $3\!+\!3\alpha$}} &\bl &
{\mbox{\tiny $1\!+\!\alpha$}}
\\\bl & {\mbox{\tiny
$2 \!+\!\alpha$}} \\ {\mbox{\tiny$1\!+\!\alpha$}}&{\mbox{\tiny 1}} }}}
\end{equation}
From this we conclude that
$v_{(4,2,0;2)} =(3+3\alpha)(2+\alpha)(1+\alpha)^2$.

\begin{theorem}For any superpartition $\La$,
 \begin{equation}\|J_\La\|^2:=\LL\olw{J_\La}|\orw{J_\La}\RR={\alpha^{\underline{\overline{\La}}}}\prod_{s\in\mathcal{B}\La}h^\La_{(\alpha)}(s) \,h^{(\alpha)}_\La(s).\end{equation}Furthermore,
 \begin{equation}\|P_\La\|^2:=\LL\olw{P_\La}|\orw{P_\La}\RR
={\alpha^{\underline{\overline{\La}}}}\prod_{s\in\mathcal{B}\La} \frac{h^{(\alpha)}_\La(s)}{h^\La_{(\alpha)}(s)}
={\alpha^{\underline{\overline{\La}}}}\, \prod_{s \in \La} \frac{h^{(\alpha)}_\La(s)}{h^\La_{(\alpha)}(s)} \, .\end{equation}
\end{theorem}
\begin{proof}Set $n=|\La|$ and $m=\underline{\overline{\La}}$.
According to Proposition 31 in \cite{DLMadv}, which is a consequence of the duality property of Jack polynomials given in  \eqref{PropgI}, we have
\begin{equation}\label{EqNormP}\LL\olw{P_\La}|\orw{P_\La}\RR=\alpha^{m+\ell_{n,m}}\frac{v_{\La'}(1/\alpha)}{v_\La(\alpha)}.\end{equation}
But {since} $\orw{J_\La}=v_\La\,\orw{P_\La}$,
{it readily follows that}
\begin{equation}\label{EqNormJ}\LL\olw{J_\La}|\orw{J_\La}\RR=\alpha^{m+\ell_{n,m}}{v_{\La'}(1/\alpha)}{v_\La(\alpha)}.\end{equation}Now, exploiting Lemma \ref{lemmahook} and the obvious property $(\mathcal{B}\La)'=\mathcal{B}(\La')$, we get
\begin{multline}v_{\La'}(1/\alpha)=\prod_{(i,j)\in \mathcal{B}(\La')}h^{\La'}_{(1/\alpha)}(i,j)
=\prod_{(i,j)\in( \mathcal{B}\La)'}\frac{1}{\alpha}h^{(\alpha)}_\La(j,i)=\frac{1}{\alpha^{n-\binom{m}2}}\prod_{(i,j)\in { \B\La}}h^{(\alpha)}_\La(i,j)\qquad\end{multline}
The above expressions for $\|J_\La\|^2$ and $\|P_\La\|^2$ follow by  substituting  the latter equation into \eqref{EqNormJ} and \eqref{EqNormP} respectively
and using $\ell_{n,m}=n-m(m-1)/2$.
That $\mathcal{B}\La$ can be replaced by $\La $ in the expression of $\|P_\La\|^2$ follows from the identity
\begin{equation}h_\La^{(\alpha)}(s)=l_{\La^\circledast}(s)+\alpha(a_{\La^*}(s)+1)=l_{\La^*}(s)+1+\alpha a_{\La^\circledast}(s)=h^\La_{(\alpha)}(s)
\end{equation}
whenever $s$ belongs to both a fermionic row and a fermionic column (i.e., $s\in\La/\mathcal{B}\La=\mathcal{F}\La$). \end{proof}

 \begin{appendix}

\section{Proofs of Propositions~\ref{PropgI} and \ref{PropgII}}
\label{appenB}

The proofs of Propositions~\ref{PropgI} and \ref{PropgII}
 rely on the relation between
 Jack polynomials in superspace
and non-symmetric Jack polynomials presented in \cite[Sect. 9]{DLMcmp2}.

The non-symmetric Jack polynomials, $E_{\eta}(x;\alpha)$, are indexed
by compositions $\eta \in \mathbb Z^N_{\geq 0}$ with $N$ parts
(some of them possibly equal to zero). They were first studied systematically
 in \cite{Op} although
they had appeared earlier in the physics litterature
as eigenfunctions of the
commuting Dunkl-type operators \cite{Ber}
\begin{equation}
{\mathcal D}_i = \alpha x_i \frac{\partial}{\partial x_i} + \sum_{k<i} \frac{x_i}{x_i-x_k} (1-K_{i,k}) + \sum_{i<k \leq N} \frac{x_k}{x_i-x_k} (1-K_{i,k}) + 1-i \, ,
\end{equation}
where $K_{i,k}$ is the operator that exchanges the variables $x_i$ and $x_k$.
The non-symmetric Jack polynomial $E_{\eta}(x;\alpha)$
can be characterized as the unique
polynomial, whose coefficient of
$x^{\eta}$ is equal to 1, such that
\begin{equation}
{\mathcal D}_i \, E_{\eta}(x;\alpha) =  \bar \eta_i \, E_{\eta}(x;\alpha) \qquad
 \forall i=1,\dots,N,
\end{equation}
where the eigenvalue $\bar \eta_i$
is given by
\begin{equation}
\bar \eta_i = \alpha \eta_i - \#\{ k < i \, | \, \eta_k \geq \eta_i \} -
 \#\{ k > i \, | \, \eta_k > \eta_i \}.
\end{equation}
The following properties of non-symmetric Jack polynomials
\cite{Knop} will prove to be important:
\begin{equation}
K_{i,i+1} \, E_{\eta}(x;\alpha) = E_{\eta}(x;\alpha) \qquad {\rm if} \quad
\eta_i=\eta_{i+1}  ,
\end{equation}
and
\begin{equation} \label{symnonsym}
\left[ K_{i,i+1} +\frac{1}{(\bar \eta_i - \bar \eta_{i+1})} \right]
 E_{\eta'}(x;\alpha)
=  E_{\eta}(x;\alpha)
 \qquad {\rm if} \quad
\eta_i>\eta_{i+1}  ,
\end{equation}
where $\eta'=(\eta_1,\dots,\eta_{i-1},\eta_{i+1},\eta_i,\eta_{i+2},
\dots,\eta_N)$.

For our purposes it will be convenient to associate a diagram to
$\eta$ given by the set of cells in $\mathbb Z^2_{\geq 1}$ such that
$1 \leq i \leq N$ and $1 \leq j \leq \eta_i$.  For instance,
if $\eta=(0,1,3,0,0,6,2,5)$, the diagram of $\eta$ is
\begin{equation}
{\tableau[scY]
{\bl \bullet \\  \\ && \\ \bl \bullet \\ \bl \bullet \\  & & & & &  \\
 & \\
& & & &  \\ }}
\end{equation}
where a $\bullet$ represents an entry of length zero.

Suppose $i_1,\dots,i_p$ are distinct integers between 1 and $N$.
It is known \cite{McAnally} that the
non-symmetric Jack polynomials satisfy the following
Pieri-type expansion
\begin{equation} \label{pierinonsym}
x_{i_1} \dots x_{i_p}  \, E_{\eta}(x;\alpha) = \sum_{\nu \in {\mathbb J}_{N,p}} c_{\eta \nu}
 \, E_{\nu}(x;\alpha),
\end{equation}
for certain coefficients $c_{\eta \nu}$.
The set ${\mathbb J}_{N,p}$ is most easily described in terms of
the diagram of $\eta$.  A cell is
first added to each of the $p$
rows $i_1,\dots, i_p$ of $\eta$ to form a new diagram.
Then ${\mathbb J}_{N,p}$ consists of all the
rearrangements of the rows of the new diagram
such that the rows with a cell added can only move downwards or stay
stationnary, while the remaining rows can only move upwards or stay
stationnary.  For instance, if $p=2$, $i_1=2$, $i_2=3$ and $\eta=(3,1,3,0)$
we have that ${\mathbb J}_{N,p}$ consists of the following diagrams
\begin{equation}
{\small{\tableau[scY]
{&& \\  & \tf\\ &&& \tf \\ \bl \bullet  }}}
\qquad  \qquad
{\small{\tableau[scY]
{&& \\  & \tf\\ \bl \bullet \\  &&& \tf  }}}
\qquad \qquad
{\small{\tableau[scY]
{&& \\ \bl \bullet \\ &&& \tf \\    & \tf }}}
\qquad \qquad
{\small{\tableau[scY]
{&& \\ \bl \bullet \\ & \tf \\    &&& \tf }}}
\end{equation}
where the cells with thick frames correspond to the rows with a cell added.

 Given a superpartition $\Lambda=
(\Lambda_1,\dots,\Lambda_m;\Lambda_{m+1},\dots,\Lambda_{N})$, define
$\tilde \Lambda$ to be the composition
\begin{equation}
\tilde \Lambda := (\Lambda_m,\dots,\Lambda_1,\Lambda_N,\dots,\Lambda_{m+1}) \, .
\end{equation}
It was established in \cite[Eq. 107 and Theo. 41]{DLMcmp2} that the Jack polynomials
in superspace can be obtained from the
non-symmetric
Jack polynomials through the following relation:
\begin{equation} \label{jackinnonsym}
{P}_{\Lambda} = \frac{(-1)^{m(m-1)/2}}{n_{\La}!}
\sum_{w \in S_N} {\mathcal K}_{w} \, \theta_1 \cdots \theta_m \,
E_{\tilde \Lambda}(x;\alpha) \, ,
\end{equation}
where we recall that $n_{\La}!$ is defined in \eqref{eqmono}, and where
${\mathcal K}_{w}$ is such that
\begin{equation}{\mathcal K}_{w} f(x_1,\dots,x_N;\theta_1,\dots,\theta_N)=
 f(x_{w(1)},\dots,x_{w(N)};\theta_{w(1)},\dots,\theta_{w(N)})
\end{equation}
on any polynomial $f(x_1,\dots,x_N;\theta_1,\dots,\theta_N)$ in $x$ and
$\theta$.

Note that the composition $\tilde \Lambda$ is of a very special form.  Its
first $m$ rows (resp. last $N-m$ rows)
are strictly increasing (resp. weakly increasing).
Diagrammatically, it is made of two partitions (the first of
which having no
repeated parts)
drawn in the French
notation (largest row at the bottom).
For instance if $\Lambda=(3,1,0;5,3,3,0,0)$, we have
$\tilde \Lambda=(0,1,3,0,0,3,3,5)$ whose diagram is given by
\begin{equation}
{\tableau[scY]
{\bl \bullet \\  \\ && \\ \bl \bullet \\ \bl \bullet \\  & &  \\
 & & \\
& & & &  \\ }}
\end{equation}
The first $m$ rows (resp. last $N-m$ rows)
of $\tilde \Lambda$ will be said to be fermionic (resp. bosonic).

We can now proceed to the proof of Propositions~\ref{PropgI}
and \ref{PropgII}.

\begin{proof}[Proof of Proposition~\ref{PropgI}]
   Let
\begin{equation}
 \mathcal O_{sym} =  \sum_{w \in S_N} {\mathcal K}_{w}  \, \theta_1 \cdots \theta_m.
\end{equation}
It is easy to see that if $f \in \mathbb Q[x_1,\dots,x_N](\alpha)$ then
$\mathcal O_{sym} \, K_{i,i+1} \, f=
-\mathcal O_{sym} \, f$
if $i=1,\dots,m-1$ and $\mathcal O_{sym} \, K_{i,i+1} \, f= \mathcal O_{sym} \, f$
if $i=m+1,\dots,N-1$.  Using \eqref{symnonsym},
we can thus deduce that
\begin{equation} \label{EqsymP}
P_{\La} \propto \mathcal O_{sym} \, E_{\eta} ,
\end{equation}
whenever the fermionic rows of $\eta$ (its first $m$
entries) are a rearrangement of $\La_1,\dots,\La_m$ and
its bosonic rows  (its last $N-m$ entries)
are a rearrangement
of $\La_{m+1},\dots,\La_N$.

We will now use the Pieri-type rule given in
\eqref{pierinonsym} to show that the expansion
\begin{equation}
e_n \, P_{\La} = \sum_{\Om} g_{\La, (1^n)}^{\Om} \, P_{\Om}
\end{equation}
is such that the
coefficient $g_{\La, (1^n)}^{\Om}$ is non-zero only if
$\Om/\La$ is a vertical $n$-strip.
We have
\begin{equation} \label{EqenP}
e_n \, P_{\La} \propto
\sum_{i_1<\dots<i_n} x_{i_1} \cdots x_{i_n}  \mathcal O_{sym} \, E_{\tilde \La}
= \mathcal O_{sym} \sum_{i_1<\dots<i_n} x_{i_1} \cdots x_{i_n}  \, E_{\tilde \La} , ,
\end{equation}
where $\sum_{i_1<\dots<i_n} x_{i_1} \cdots x_{i_n}$ commutes with $\mathcal O_{sym}$
since it is a symmetric function in $x_1,\dots,x_N$.
Given a composition $\eta$ whose first $m$ entries are all distinct,
let $\Om_{\eta}=(\Om_\eta^a;\Om_\eta^s)$ be the superpartition such that
$\Om_\eta^a$ and $\Om_\eta^s$ are obtained respectively
by rearranging
the first $m$ entries of $\eta$ and the last $N-m$ entries of
$\eta$.  We thus simply need to show that the compositions $\eta$ such that
$E_{\eta}$ appear in $x_{i_1} \cdots x_{i_n}  \, E_{\tilde \La}$ are such
that $\Omega_{\eta}/\La$ is a vertical $n$-strip.

We know from the rule given after
\eqref{pierinonsym} that $\eta$ is obtained from $\tilde \La$
by adding $n$ cells in distinct rows and then rearranging the rows.
It is thus clear that $\Om_\eta^*/\La^*$ is a vertical $n$-strip.
Suppose that $\omega$ is obtained from $\tilde \La$ by adding $n$ cells
in distinct rows.
It is easily seen that $\Om_{\omega}^{\circledast}/\La^{\circledast}$ is a vertical $n$-strip
in that case.
We will now see that if $\eta$ is obtained by rearranging the rows of $\omega$
then $\Om_{\eta}^{\circledast}/\La^{\circledast}$ is still
a vertical $n$-strip.  First observe from \eqref{EqsymP}
that
the only rearrangments that matter are those
that send a fermionic (resp. bosonic) row into a bosonic (resp. fermionic) row.
Since fermionic rows lie above bosonic ones, for a bosonic row of $\omega$ to become fermionic, the rule given after
\eqref{pierinonsym} imposes that
the size of that row needs to be the same in $\omega$ and in $\tilde \La$.  Therefore this new
fermionic row differs from the old bosonic row of $\tilde \La$ only by a circle, and thus when
comparing the the two rows in $\Om_{\eta}^{\circledast}/\La^{\circledast}$ we get a difference of one.
The rule given after
\eqref{pierinonsym} imposes similarly that
for
a fermionic row of $\omega$ to become bosonic, the size of that row needs
to be larger (by one) in $\omega$ than in $\tilde \La$.  Hence, when
comparing the two rows in $\Om_{\eta}^{\circledast}/\La^{\circledast}$ we get that
they are of the same size (a circle was lost while a square was gained).
We thus have that $\Om_{\eta}^{\circledast} \subseteq \La^{\circledast}$
with
the length of the
rows of $\Om_{\eta}^{\circledast}$ and $\La^{\circledast}$ never differing
by more than one.  Consequently, $\Om_{\eta}^{\circledast}/\La^{\circledast}$
is a vertical $n$-strip.
\end{proof}

\begin{proof}[Proof of Proposition~\ref{PropgII}]
We will prove the equivalent statement that
\begin{equation}
\tilde e_n \, P_{\La} = \sum_{\Om} g_{\La, (0;1^n)}^{\Om} \, P_{\Om}
\end{equation}
is such that the
coefficient $g_{\La, (0;1^n)}^{\Om}$ is non-zero only if
$\Om/\La$ is a vertical $\tilde n$-strip.

We use again the Pieri-type formula for non-symmetric Jack polynomials to get
\begin{equation}
\tilde e_n \, P_{\La} \propto
\sum_{\substack{i, i_1<\dots<i_n\\ i\not\in\{i_1,\cdots, i_n\}}}  \theta_i
\, x_{i_1} \cdots x_{i_n} \mathcal O_{sym} \, E_{\tilde \La}
= \mathcal O_{sym} \sum_{\substack{i, i_1<\dots<i_n\\ i\not\in\{i_1,\cdots, i_n\}}} \theta_i x_{i_1} \cdots x_{i_n}  \, E_{\tilde \La} .
\end{equation}
Following the argument given after \eqref{EqenP}, it is immediate that
we have again that if $\Omega$ appears in $\tilde e_n \, P_{\La}$ then $\Om^*/\La^*$ is a vertical $n$-strip.
It remains to show that  the terms $\Omega$ occurring
in $\tilde e_n \, P_{\La}$ are such
that $\Omega^{\circledast}/\La^{\circledast}$ is a vertical $(n+1)$-strip.  This is somewhat more subtle and we
will use a different route to obtain the result.

Let $d= \sum_{i=1}^{\infty} \theta_i \partial_{x_i} $ and $d^{\perp}=
\sum_{i=1}^{\infty} x_i \partial_{\theta_i}$.  The operators $d$ and $d^{\perp}$ are adjoint of each
other
with respect to the scalar product  \eqref{defscalprodcomb},
namely:
\begin{equation}
\LL d^{\perp} \, f\, | g \, \RR = \LL f \, | \, d
\, g\RR,
\end{equation}
for every symmetric functions in superspace $f$ and $g$.
It is easy to see that the expansion
\begin{equation} \label{equad}
d \, m_{\La}= \sum_{\Om} u_{\La \Om}\, m_{\Om}
\end{equation}
is such that $u_{\La \Om} \neq 0$ only if $\Om^{\circledast}=\La^{\circledast}$ ($\theta_i \partial_{x_i}$
removes a square from a bosonic row and changes the row into a fermionic one).
Similarly,
\begin{equation} \label{equadd}
d^{\perp} \, m_{\La}= \sum_{\Om} \tilde{u}_{\La \Om}\, m_{\Om}
\end{equation}
is such that $\tilde{u}_{\La \Om} \neq 0$ only if $\Om^{\circledast}=\La^{\circledast}$ ($x_i \partial_{\theta_i}$
adds a square to a fermionic row and changes the row into a bosonic one).

By \eqref{Ptriangular}, we have that
$P_{\La} = \sum_{\Om \leq \La} c_{\La \Om} \, m_{\Om}$.  By
definition of the dominance order on superpartions, we get in
particular that $m_{\Omega}$ occurs in $P_{\La}$ only if
$\Om^{\circledast} \leq \La^{\circledast}$.
Thus, from \eqref{equad} and \eqref{equadd}, we have
\begin{equation} \label{dsurLa}
d\, P_{\La} =  \sum_{\Om^{\circledast} \leq \La^{\circledast}} b_{\La \Om} \, P_{\Om} \qquad {\rm and} \qquad
d^{\perp}\, P_{\La} =  \sum_{\Om^{\circledast} \leq \La^{\circledast}}
\tilde{b}_{\La \Om} \, P_{\Om}.
\end{equation}

Using the adjointness of $d$ and $d^{\perp}$ we have
\begin{equation}
\LL P_{\Gamma} \, | \,  d \, P_{\La}\RR
= \LL d^{\perp}  P_{\Gamma} \, | \, P_{\La}\RR .
\end{equation}
Then, from the orthogonality of the Jack polynomials in superspace,
we have from \eqref{dsurLa} that the left-hand
side
is zero unless $\Gamma^{\circledast} \leq \La^{\circledast}$
while the right-hand side is zero unless $\Gamma^{\circledast}
\geq \La^{\circledast}$.  Therefore the expressions are
zero unless $\Gamma^{\circledast} =
 \La^{\circledast}$, that is,
\begin{equation} \label{dsurLaII}
d \, P_{\La}= \sum_{\Om^{\circledast} = \La^{\circledast}} b_{\La \Om}\, P_{\Om}.
\end{equation}

Now, an easy computation gives that
$d \, e_{n+1} = \tilde e_n$.
Since $d$ is a derivative,  we have
\begin{equation}
d(e_{n+1} \, P_{\La}) =  \tilde e_n \, P_{\La} + e_{n+1} \, d (P_{\La}) ,
\end{equation}
and thus
\begin{equation}
\tilde e_n \, P_{\La} =  e_{n+1} \, d (P_{\La}) - d(e_{n+1} \, P_{\La}).
\end{equation}
Since, by \eqref{dsurLaII},
all the terms $\Gamma$ that occur in $d\, P_{\La}$ are such that
$\Gamma^{\circledast}=\La^{\circledast}$ we have by Proposition~\ref{PropgI}
(in the form demonstrated above)
that
all the terms $\Om$ that occur in $e_{n+1} \, d (P_{\La})$ are such that
$\Omega^{\circledast}/\La^{\circledast}$ is a vertical $(n+1)$-strip.  Similarly,
all the terms $\Gamma$ that occur in $e_{n+1} \, P_{\La}$ are such that
$\Gamma^{\circledast}/\La^{\circledast}$ is a vertical $(n+1)$-strip and thus
by \eqref{dsurLaII} all the terms $\Omega$ that occur in $d(e_{n+1} \, P_{\La})$ are such
that  $\Omega^{\circledast}/\La^{\circledast}$ is a vertical $(n+1)$-strip.
Therefore, all the terms $\Omega$ that occur in $\tilde e_n \, P_{\La}$ are such
that $\Omega^{\circledast}/\La^{\circledast}$ is a vertical $(n+1)$-strip.

\end{proof}

\section{Orderings on superpartitions and Jack polynomials in superspace}
\label{appen1}

Let $\leq$ and $\trianglelefteq$ refer respectively
to the orders on superpartitions defined in \eqref{eqdeforder1} and
\eqref{eqorder2}.
The Jack polynomials in superspace were
defined in \cite{DLMcmp2}
as in
Theorem~\ref{TheoEigenJack} but with the order
$\trianglelefteq$ instead of  $\leq$.
We will show in this appendix that the two orders
lead to the same family of Jack polynomials in
superspace.

The Jack polynomials are known \cite{DLMcmp2}
to be such that
\begin{equation}
{\mathcal I} \,
P_{\La} = \epsilon'_\Lambda \, P_{\La},\qquad \text{and}\qquad
P_{\La}=m_{\La} + \sum_{\Omega
  \triangleleft \La } c_{\La \Omega} \, m_{\Omega},\end{equation}
where $\alpha\epsilon'_\La=\epsilon_\La-m(m-1)/2$,
with $\epsilon_\La$ being defined in \eqref{epsi}.
Note that the relation between ${\mathcal I} $ and $\Delta$ is given
in \eqref{DDrel}. A crucial step in the
derivation of those results was to show that
\begin{equation}
{\mathcal I} \, m_{\La} =  \epsilon'_\Lambda \,  m_{\La}
+ \sum_{\Omega \triangleleft \La} d_{\La \Omega} \, m_{\Omega} \, .
\end{equation}
Our main task here is to prove the stronger statement
(given that $\Gamma \leq \Lambda$ implies
$\Gamma \trianglelefteq \Lambda$)
\begin{equation}\label{eqI1}
{\mathcal I} \, m_{\La} =  \epsilon'_\Lambda \,  m_{\La}
+ \sum_{\Omega < \La} b_{\La \Omega} \, m_{\Omega},
\end{equation}
where we emphasize that the order in the sum is now the order $\leq$.
Since $\epsilon'_{\Gamma} \neq \epsilon'_{\La}$ if
$\Gamma \triangleleft \La$ (see \cite{DLMcmp2}),
 \eqref{eqI1} ensures that  $P_{\La}=m_{\La} + \sum_{\Omega
  < \La } c_{\La \Omega} \, m_{\Omega}$, which is the result we are
trying to establish.
Suppose otherwise that there exists a  $c_{\La \Gamma} \neq 0$
with $\Gamma \triangleleft
\Lambda$ and $\Gamma \not < \Lambda$.
Pick $\Gamma$ to be such that there is
no  $c_{\La \Omega} \neq 0$ with $\Omega \triangleleft
\Lambda$, $\Omega \not < \Lambda$ and $\Omega > \Gamma$.
Then we get
the contradiction that the coefficient of $c_{\La \Gamma} \, m_{\Gamma}$
in ${\mathcal I} \,
P_{\La}$ is $\epsilon'_{\Gamma} \neq \epsilon'_{\La}$.

We now prove \eqref{eqI1}.  We only need to show that there does not exist a  $\Gamma$
such that $b_{\La \Gamma} \neq 0$ with
$\Gamma^{\circledast} \not < \La^{\circledast}$.  In \cite[Eq. 87]{DLMcmp2},
it is shown that the coefficient of $\theta_1 \cdots \theta_m$ in
${\mathcal I} \, m_{\La}$ is given, up to a factor, by\footnote{Minor
misprints in \cite[Eq. 87]{DLMcmp2}
 are corrected here.} (using $\beta=1/\alpha$)
 \begin{equation}
\left[ \sum_{i=1}^m \La_i -\beta m(m-1)  \right] x^{\La} +
\frac{\beta}{n_{\La}!} \sum_{w \in S_m} (-1)^{{\rm sgn}(w)} K_w \sum_{i=1}^m
\sum_{j=m+1}^N \frac{x_j}{(x_i-x_j)} (1-K_{ij}) x^{\La}  ,
\end{equation}
where $n_{\La}!$ is defined in \eqref{eqmono}.  Observe that a  term
$x^\eta$ in the resulting expression contributes to the coefficient of
$m_{\Gamma}$,
where $\Gamma=(\Gamma^a;\Gamma^s)$ is such that $\Gamma^a$ is the reordering of the
first $m$
entries of $\eta$ and $\Gamma^s$ is the reordering of the remaining
entries of $\eta$.  Since the $\eta$'s that can appear in the resulting expression
differ from $\La$ in at most two entries, it suffices to consider the
two-variable case.  Let $\Lambda=(a;b)$ (for the cases $(a,b;)$ and $(a,b)$, the conclusion is immediate).
We have
\begin{equation} \label{eqA3}
\frac{x_2}{(x_1-x_2)} (1-K_{12}) \, x_1^a x_2^b
=
\left \{
\begin{array}{l}
x_1^{a-1}x_2^{b+1}+x_1^{a-2}x_2^{b+2} +\cdots + x_1^{b-1}x_2^{a+1}+x_1^b
x_2^{a} \quad \text{if } a>b\\
\\
x_1^{b-1}x_2^{a+1}+x_1^{b-2}x_2^{a+2} +\cdots + x_1^{a-1}x_2^{b+1}+x_1^a
x_2^{b}  \quad \text{if } b>a.
\end{array}
\right.
\end{equation}
In the case where $a>b$ we have $\Lambda^{\circledast}=(a+1,b)$ and it is easy
to see that it is larger in the dominance order than every term of the form $(z+1,y)$
where $x_1^z x_2^y$ appears in \eqref{eqA3}.   Similarly,
in the case where $a<b$ we have $\Lambda^{\circledast}=(b+1,a)$ and we
see that it is larger or equal in the dominance order to every term of the form $(z+1,y)$
where $x_1^z x_2^y$ appears in \eqref{eqA3}.  This implies that in the
two-variable case 
every term $m_{\Gamma}$ present in the action of ${\mathcal
  I}$ on $m_{\La}$ is such that $\Gamma^{\circledast} \leq \La^{\circledast}$.
As previously mentioned, the general case follows immediately.

\section{Another combinatorial expression for the evaluation formula}
\label{appD}

The evaluation formula of
Theorem~\ref{TheoSpecialI} is expressed in terms of the
 skew diagram $\S \La$.  It is
possible to reexpress this evaluation formula directly in terms of the diagram of $\La$.
This alternative expression involves what we will call the
shadow of a cell in analogy with Viennot's shadow in \cite{Vie}.

The shadow of $s$ is made of all the cells weakly south-east of it, that is,
the cells in the shadow of $s=(i,j)$ are the cells $s'=(i',j')$ such that $i' \geq i$ and $j' \geq j$.
In the diagram of $\La$, we place a $\bullet$  in the $j$-th cell or circle of the $(m-j+1)$-th circled row (from top to bottom). We then define
\begin{align}
\#_\circ s =& \text {   the number of circles in the shadow of $s$} ,
\nonumber \\
\#_\bullet s=& \text {  the number of $\bullet$ is the shadow of $s$.}
\end{align}
For instance, for the superpartition $(5,3,1;2,2,2,1)$, the position of the $\bullet$ and the shadow of cell $s=(3,1)$ (indicated by $x$'s)
are as follows:
\begin{equation} {\tableau[scY]
{&  & \bullet &  & & \bl \tcercle{}\\
& \bullet & &\bl\tcercle{} \\
x& x & \bl x & \bl x & \bl x & \bl x & \bl x\\
x& x & \bl x & \bl x & \bl x & \bl x \\
x& x & \bl x & \bl x & \bl x
\\ \bullet & \bl\tcercle{} & \bl x & \bl x
\\ x & \bl x & \bl x \\
 \bl x & \bl x \\
 \bl x}}
\end{equation}
so that $\#_\circ(3,1)= 1$ and $\#_\bullet (3,1)= 1$.

Recall that the definitions of arm-colengths and leg-colengths
can be found in Section~\ref{sectintro}.
\begin{proposition} Let
\begin{equation}\label{spe}
\Sp_{N,m}\,[J_{\La}]\,= \prod_{s \in
\mathcal{B}\La }{f_\La(s)}
\end{equation}
with  $f_{\La}(s)$  given by
\begin{align}
f_\La(s)&= N-l_{\La^*}'(s)+\alpha(a_{\La^*}'(s)+\#_\circ s)\quad \,\text{if}\quad \La^{\circledast}_i-\La_i^*=1\\
&= N-l_{\La^*}'(s)-\#_\bullet s +\alpha\,
a_{\La^*}'(s)\phantom{()}\quad \text{if}\quad  \La^{\circledast}_i-\La_i^*=0
 \end{align}
 where $s=(i,j)$.
\end{proposition}




For instance, if $\Lambda=(3,1,0;4,2,1)$, filling the cells
$s \in \mathcal{B}\La$ with
the values $f_{\La}(s)$  gives (using $\bar{k} = N-k$):
\begin{equation} {\small{\tableau[mcY]{{\mbox{\tiny
$\bar{3}$}}&{\mbox{\tiny $\bar{2}\!+\!\alpha$}} &{\mbox{\tiny $\bar{1}\!+\!2\alpha$}}&
{\mbox{\tiny $\bar{0}\!+\! 3\alpha$}} \\& &{\mbox{\tiny
$\bar{1}\!+\!3 \alpha$}}&\bl\gcercle\\{\mbox{\tiny $\bar{4}$}}&{\mbox{\tiny
$\bar{3}\!+\!\alpha$}}\\ & \bl\gcercle\\{\mbox{\tiny $\bar{5}$}}\\  \bl \gcercle}} }
\end{equation}
In other words, we have:
\begin{equation}\label{D4}
\Sp_{N,3}\,[J_{(3,1,0;4,2,1)}]= (N-3)(N-2+\alpha)(N-1+2\alpha)(N-1+3\alpha)(N-4)(N-3+\alpha)(N-5).
\end{equation}

This can be compared with the result obtained from filling the skew tableau $\S\La$ with the values $b_{(3,1,0;4,2,1)}^{(\alpha,N)}(s)$ defined in
Theorem~\ref{TheoSpecialI}:
\begin{equation} {\small{\tableau[mcY]{
{\mbox{\tiny$$}}
&{\mbox{\tiny }} &{\mbox{\tiny }}&
{\mbox{\tiny $\bar{0}\!+\!3\alpha$}} \\& &{\mbox{\tiny
$\bar{1}\!+\!2 \alpha$}}&{\mbox{\tiny$\bar{1}\!+\!3\alpha$}} \\{\mbox{\tiny $ $}}&{\mbox{\tiny
$\bar{2}\!+\!\alpha$}}\\ {\mbox{\tiny$\bar{3}$}} &{\mbox{\tiny$\bar{3}\!+\!\alpha$}}
\\
{\mbox{\tiny $\bar{4}$}}\\ {\mbox{\tiny$\bar{5}$}}
}} }
\end{equation}
The resulting expression for $\Sp_{N,3}[J_{(3,1,0;4,2,1)}]$, obtained by taking the products of the entries of the filled squares, is clearly equal to
\eqref{D4}.

The relation between the two expressions for $\Sp_{N,m}[J_{\La}]$
is simply described as follows.
Notice at once that the number of filled squares is the same in the two representations: the number of squares in ${\mathcal B}\La$ is $|\La|-m(m-1)/2$, while there are $|\La|+m-m(m+1)/2$ squares in $\S \Lambda$.
Take the filling of the squares of $\La$ described by the factor $f_\La(s)$. Then replace the circles by squares, thus transforming
the Ferrers diagram of $\La$ into that of $\La^{\circledast}$. Finally, move the filled squares as follows: if the square $s$ belongs to a fermionic (resp. bosonic)
row of $\La$ then displace it to the right (resp. downward) by $\#_\circ s$
(resp. $\#_\bullet s$) units.  It is then straightforward to see that
the two evaluations coincide.

\end{appendix}


\begin{thebibliography}{99}





\bibitem{BDF}
 T.\ H.\ Baker, C.\ F.\  Dunkl, and P.\ J.\ Forrester,\emph{ Polynomial eigenfunctions of the Calogero-Sutherland-Moser
models with exchange terms}, pages 37--42 in J.\ F.\ van Diejen and L.\ Vinet, \emph{ Calogero-Sutherland-Moser Models},  CRM Series in Mathematical Physics, Springer (2000).

\bibitem{Baker}
T.~H.~Baker and P.~J. Forrester, {\it The Calogero-Sutherland model and polynomials with prescribed symmetry}, Nucl.\ Phys.\  B 492 (1997),  682--716.



\bibitem{Ber}
D.~Bernard, M.~Gaudin, F.~D.~Haldane and V.~Pasquier,
\emph{Yang-Baxter equation in long range interacting system},
J.\ Phys.\ A {\bf A26},  (1993) 5219-5236.

\bibitem{BTW}
L.~Brink, A.~Turbiner and N.~Wyllard, \emph{Hidden Algebras of the
(super) Calogero and Sutherland models}, J.\ Math.\ Phys.\  {\bf 39}
(1998), 1285--1315.


\bibitem{CL}
S.\ Corteel and J.\ Lovejoy, \emph{Overpartitions},
Trans.\ of the Am.\ Math.\ Soc.\ {\bf 356} (2004) 1623--1635.




\bibitem{DLMnpb}
P.~Desrosiers, L.~Lapointe and P.~Mathieu, \emph{Supersymmetric
Calogero-Moser-Sutherland models and Jack superpolynomials}, Nucl.\
Phys.\ {\bf B606} (2001),  547--582.


\bibitem{DLM2}
P.~Desrosiers, L.~Lapointe and P.~Mathieu, \emph{Jack
superpolynomials, superpartition ordering and determinantal
formulas}, Commun.\ Math.\ Phys.\ {\bf 233} (2003), 383--402.


\bibitem{DLMcmp2}
P.~Desrosiers, L.~Lapointe and P.~Mathieu, \emph{Jack polynomials in
superspace}, Commun.\ Math.\ Phys.\ {\bf 242} (2003),  331--360.



\bibitem{DLMjalgcomb}
P.~Desrosiers, L.~Lapointe and P.~Mathieu, \emph{Classical symmetric functions
  in  superspace}, J.\ Alg.\ Comb.\ {\bf 24} (2006),  209--238.

\bibitem{DLMadv}
P.~Desrosiers, L.~Lapointe and P.~Mathieu, \emph{Orthogonality of Jack polynomials in superspace}, Adv. Math. {\bf 212} (2007), 361--388.



\bibitem{Dunkl98}
C.\ F.\ Dunkl, \emph{Orthogonal Polynomials of Types A and B and Related Calogero Models},
Commun.\  Math.\ Phys.\ 197 (1998), 451--487.





\bibitem{McAnally}
P.~J.~Forrester, D.~S.~McAnally and Y.~Nikoyalevsky, {\it On the evaluation formula for Jack polynomials with prescribed symmetry}, J.\ Phys.\  {\bf A 34} (2001), 8407-8424.


\bibitem{Knop}F.~Knop and S.~Sahi, {\it A recursion and a combinatorial formula for Jack polynomials}, Invent.\ Math.\ {\bf 128} (1997), 9--22.



\bibitem{LLN}
L. Lapointe, Y. Le Borgne and P. Nadeau
\emph{A normalization formula for the Jack polynomials in superspace and an identity on partitions},
Electronic J. Comb.  {\bf 16} (2009) Article \#R70.



\bibitem{Mac}
               {I.~G.~ Macdonald},
                {Symmetric functions and {H}all polynomials},
2nd ed., Clarendon Press, 1995.


\bibitem{Op}
E.M. Opdam, \emph{Harmonic analysis for certain representations of graded Hecke algebras}, Acta Math. {\bf 175} (1995) 75-121.



\bibitem{Pak}
I.\ Pak, \emph{Partition bijections, a survey},  Ramanujan J.\ {\bf 12} (2006) 5--75.




\bibitem{SS}
B.~S. Shastry and B.~Sutherland, \emph{Superlax pairs and infinite
symmetries in the $1/r^2$ system}, Phys.\ Rev.\ Lett.\ {\bf 70}(1993), 4029--4033.





\bibitem{Stan}                 R.~P.~Stanley,
               \emph{Some combinatorial properties of Jack symmetric
functions}, Adv.\ Math.\ {\bf77} (1988), 76--115.


\bibitem{Vie} G. Viennot, \emph{Une forme g\'eom\'etrique de la correspondance
de Robinson-Schensted}.
In ``Combinatoire et repr\'esentation du groupe sym\'etrique
(Actes Table Ronde CNRS, Univ. Louis-Pasteur Strasbourg, Strasbourg, 1976)'',
29--58.  Lecture Notes in Mathematics 579, Springer, 1977.






\end{thebibliography}
\end{document}